\documentclass{mcom-l}

\usepackage{array}
\usepackage{amssymb}
\usepackage{amsmath}
\usepackage{amsthm}
\usepackage{graphicx}
\usepackage{bm}
\usepackage{pgfplots}
\usepackage{caption}
\usepackage{url}
\usepackage{comment}
\usepackage{longtable}

\newcommand{\blackboard}[1]{\mathbb{#1}}

\newcommand{\Q}{\blackboard{Q}}

\newcommand{\D}{\Delta}

\newcolumntype{H}{>{\setbox0=\hbox\bgroup}c<{\egroup}@{}}

\setlongtables

\newtheorem{thm}{Theorem}[section]
\newtheorem{cor}[thm]{Corollary}
\newtheorem{conj}[thm]{Conjecture}

\theoremstyle{definition}
\newtheorem{defin}[thm]{Definition}

\begin{document}

\title[Unconditional Class Group Tabulation to $2^{40}$]{Unconditional
  Class Group Tabulation of Imaginary Quadratic Fields to $|\Delta| < 2^{40}$}

\author{A.~S.\ Mosunov}
\address{University of Waterloo \\
                200 University Ave W,
                Waterloo, Ontario \\ Canada N2L 3G1}
\email{amosunov@uwaterloo.ca}
\thanks{The first author's research was supported by ``Alberta Innovates Technology Futures'', Canada.}

\author{M.~J.\ Jacobson, Jr.}
\address{University of Calgary \\
                2500 University Drive NW,
                Calgary, Alberta \\ Canada T2N 1N4}
\email{jacobs@cpsc.ucalgary.ca}
\thanks{The second author's research is supported by NSERC of Canada.}


\begin{abstract}
We present an improved algorithm for tabulating class groups of
imaginary quadratic fields of bounded discriminant.  Our method uses
classical class number formulas involving theta-series to compute the
group orders unconditionally for all $\Delta \not \equiv 1 \pmod{8}.$ The
group structure is resolved using the factorization of the group
order.  The $1 \bmod 8$ case was handled using the methods of
\cite{jacobson}, including the batch verification method based on the
Eichler-Selberg trace formula to remove dependence on the Extended
Riemann Hypothesis.  Our new method enabled us to extend the
previous bound of $|\Delta| < 2 \cdot 10^{11}$ to $2^{40}$.
Statistical data in support of a variety conjectures is presented,
along with new examples of class groups with exotic structures.
\end{abstract}

\maketitle

\section{Introduction} \label{sec:intro}

The class group of an imaginary quadratic field $\Q(\sqrt{\D})$ with
discriminant $\Delta,$ denoted by $Cl_\Delta,$ has been studied
extensively over the past two centuries.  Many things are known about
the class group. For example, if we know the \emph{class number}
$h(\Delta)$, which is defined as the size of $Cl_\Delta$, we can find
a non-trivial factor of $\Delta$. Also, from the prime factorization
of $\Delta$ we can determine the parity of $h(\Delta)$, as well as the
rank of the $2$-Sylow subgroup of $Cl_\Delta.$

However, the number of open questions about $Cl_\Delta$ most certainly
exceeds the number of answered. For example, computing the class
number is believed to be computationally difficult; it is known to be at
least as hard as integer factorization, and is currently harder.  The
heuristics of Cohen and Lenstra \cite{cohen2} allow us to make certain
predictions regarding divisibility properties of $h(\Delta)$ and the
structure of $Cl_\Delta,$ but most of these, especially with respect
to odd primes, remain unproved.  Another question of interest is to
provide tight bounds on $h(\Delta)$.  This has been answered by
Littlewood \cite{littlewood}, but the result is \emph{conditional};
that is, it depends on the \emph{Extended Riemann Hypothesis} (ERH).

Due to the lack of unconditional proof on such basic arithmetic
properties on $Cl_\Delta,$ it is of interest to provide numerical
evidence supporting the heuristics and conditional results.
Tabulating $Cl_\Delta$ for as many small discriminants as possible
provides such evidence.  The first major work on class group
tabulation is due to Buell, who in a series of papers culminating in
\cite{buell}, computed all $Cl_\Delta$ for negative $\Delta$
satisfying $|\Delta| < 2.2\cdot 10^9$. In his work, Buell gathered
statistics on Littlewood's bounds on $L(1, \chi_\Delta)$
\cite{littlewood}, Bach's bound on the size of the generators required
to produce $Cl_\Delta$ \cite{bach1}, and the Cohen-Lenstra heuristics
\cite{cohen2}. He also provided a table of first occurrences of
so-called ``exotic'' groups. These groups possess interesting group
structures, such as non-cyclic $p$-Sylow subgroups for odd primes $p$,
which according to the Cohen-Lenstra heuristics are quite rare.
Such groups are of interest, for example, in the context of class field
theory, as the towers of field extensions for them have interesting,
non-trivial properties \cite{meyer}.

The next and also most recent work of interest is due to Jacobson et
al.\ \cite{jacobson}, who used a baby-step giant-step algorithm to
tabulate all class groups to $10^{11}$
\cite[Algorithm~4.1]{buchmann}. The bound was further extended to
$2\cdot 10^{11}$ in the Master's thesis of Ramachandran
\cite{ramachandran}. The authors used Bach's averaging method in order
to determine a \emph{conditional} lower bound $h^*$ on $h(\Delta)$,
such that \mbox{$h^* \leq h(\Delta) \leq 2h^*$} \cite{bach2}. Due to the
nature of the baby-step giant-step algorithm, knowing this bound was
sufficient to be certain that the whole group was generated, assuming
the ERH. In order to eliminate the ERH dependency, they applied the
Eichler-Selberg trace formula \cite{schoof}, which relates sums of
Hurwitz class numbers to the trace of a certain Hecke operator
\cite[Formula~2.2]{jacobson}. Following Buell, the authors gathered
statistics on various hypotheses regarding $Cl_\Delta$.

In this paper, we push the feasibility limit further by tabulating
class groups for all negative $\Delta$ such that $|\Delta| < 2^{40} =
1.09951\ldots \times 10^{12}$. Using certain class number generating
functions \cite{watson}, we were able to compute all class numbers
$h(\Delta)$ for $\Delta \not \equiv 1 \pmod{8}$ via a product of two
large-degree power series; this method was inspired by that of Hart et
al.\ to tabulate all congruent numbers to $10^{12}$ \cite{hart}.
Computing the class numbers first allowed us to further achieve a
significant speedup in class group tabulation for these $\Delta$ by
only resolving the structure of possibly non-cyclic subgroups, i.e.,
for which all prime divisors of the order occur with multiplicity
greater than one.  The discriminants $\Delta \equiv 1 \pmod{8}$ were
handled separately using the previous technique of Jacobson et
al.\ \cite{jacobson}. In the end, we observed that the class groups
with $\Delta \not \equiv 1 \pmod{8}$ were computed over 4.72 times
faster than class groups with $\Delta \equiv 1 \pmod{8}.$

Unfortunately, the improved running time did not come for free, as we
were no longer able to test Bach's bound on the size of generators
required to produce the whole group \cite{bach1}. Nevertheless, we
were still able to gather extensive computational evidence in support
of Littlewood's bounds \cite{littlewood}, the Cohen-Lenstra heuristics
\cite{cohen2}, and Euler's hypothesis on idoneal numbers
\cite{kani}. We also further extended Buell's table of exotic groups
\cite{buell}.

Our paper is organized as follows. In Section \ref{sec:clnum}, we
present three formulas suitable for the tabulation of class numbers
$h(\Delta)$ with $\Delta \not \equiv 1
\pmod{8}$. Section~\ref{sec:out-of-core} is dedicated to the
out-of-core polynomial multiplication technique due to Hart et
al.\ \cite{hart}, which allows to compute the product of two large
polynomials that cannot fit into memory all together. Section
\ref{sec:clgrp} gives a brief overview of the techniques that were
used in order to tabulate $Cl_\Delta$ for $\Delta \equiv 1
\pmod{8}$. Section \ref{sec:performance} discusses the performance of our program.
In Section \ref{sec:results}, we present our numerical
results, which include statistics on various hypotheses regarding
$Cl_\Delta$ and the refined table of exotic groups. Section
\ref{sec:furtherwork} concludes the paper by giving a discussion of
various techniques, which can further accelerate the class number and
class group tabulation.

\section{Preliminaries} \label{sec:preliminaries}

Our method for computing the class numbers relies on classical
results related to binary quadratic forms.  Hence, our algorithms will
be described in the language of forms, and we will use the fact that
the ideal class group of the field $\Q(\sqrt{\D})$ of discriminant
$\Delta < 0$ is isomorphic to the group of equivalence classes of
binary quadratic forms of discriminant $\Delta.$

In particular, we consider binary quadratic forms from two different
perspectives. We use $(a, b, c)$ to denote a \emph{modern} binary
quadratic form, i.e.\ form which possesses a \emph{discriminant}
$\Delta = b^2 - 4ac$. We also use $(a, 2b, c)$ to denote a
\emph{classical} binary quadratic form of \emph{determinant} $D = b^2
- ac$, studied by Gau\ss\, \cite[Chapter 5]{gauss}.  In the first case,
the set of equivalence classes with respect to invertible integral
linear changes of variables forms a group under composition of forms.
An analogous observation can be made regarding the set of \emph{properly
primitive} (to be defined) classical quadratic forms.
In the first case, the class group of modern forms is isomorphic to
the ideal class group of $\Q(\sqrt{\Delta})$ whenever $\Delta$ is a
field discriminant (square-free integer congruent to $1 \bmod 4$ or
$4$ times a square-free number).  The second case is closely related; 
as the formula (\ref{eq:h}) suggests, the resulting group order
corresponding to determinant $D$ differs from $h(\Delta)$ by a factor
of three if $\Delta \equiv 5 \pmod{8}$, $\Delta \neq -3$ and is equal to $h(\Delta)$
otherwise.

We also require the following classifications of classical forms:
\begin{defin} \label{defin:primitive_form} \emph{\cite[\textsection 226]{gauss}}
Consider a quadratic form $(a, 2b, c)$ and its \emph{divisor} $\delta
= \gcd(a, b, c)$. Then $(a, 2b, c)$ is called \emph{primitive} if
$\delta = 1$, and \emph{derived} otherwise.
\end{defin}

\begin{defin} \label{defin:proper_form} \emph{\cite{kronecker}}
A primitive quadratic form $(a, 2b, c)$ is called \emph{uneven} if
$\gcd(a, 2b, c) = 1$, i.e., its coefficients $a$ and $c$ are not both
even; it is called \emph{even} otherwise. A derived form of divisor
$\delta$ is uneven when $(a/\delta, 2b/\delta, c/\delta)$ is uneven;
otherwise it is even.
\end{defin}

By $F(n)$ and $F_1(n)$ Kronecker denoted the total number of uneven
and even equivalence classes of forms of determinant $D = -n$,
respectively. We extend his notation by writing $\tilde{F}(n)$ and
$\tilde{F}_1(n)$ for the total number of \emph{primitive} uneven and
even equivalence classes of determinant $D = -n$, respectively. Note
that there exists a straightforward connection between $F(n)$ and
$\tilde{F}(n)$, and between $F_1(n)$ and $\tilde{F}_1(n).$ In
particular, if we write $n = g^2e$, where $e$ is square-free, then
\begin{equation} \label{eq:F_and_F_tilde}
F(n) = \sum\limits_{t\,|\,g}\tilde{F}\left(\frac{n}{t^2}\right) \,\,\, \textnormal{and} \,\,\, F_1(n) = \sum\limits_{t\,|\,g}\tilde{F}_1\left(\frac{n}{t^2}\right).
\end{equation}
To see this, observe that when $\gcd(a, b, c) = t > 1$, from every
uneven primitive form $(a/t, 2b/t, c/t)$ of determinant $-n/t^2$ we
can obtain every uneven derived form $(a, 2b, c)$ of determinant $D
= -n$ and divisor $\delta = t$. By counting all uneven primitive
forms with all uneven derived ones, we obtain $F(n)$. A similar
reasoning allows us to deduce the formula for $F_1(n)$.

\section{Class Number Tabulation Formulas} \label{sec:clnum}

We begin by considering the following Jacobi theta series:
$$
\vartheta_2(q) = 2\sum_{k = 0}^\infty q^{\left(k + \frac{1}{2}\right)^2} = 2q^{\frac{1}{4}} + 2q^{\frac{9}{4}} + 2q^{\frac{25}{4}} +  2q^{\frac{49}{4}} + \ldots\,\,\,;
$$ 
$$
\vartheta_3(q) = 1 + 2\sum_{k = 1}^\infty q^{k^2} = 1 + 2q + 2q^4 + 2q^9 + 2q^{16} + \ldots \,\,.
$$ 
In 1860, Kronecker found the connection that exists between
$\vartheta_3(q)$ and classical quadratic forms. We summarize his
result in Theorem \ref{thm:kronecker}.
\begin{thm}[{\cite{kronecker}}]\label{thm:kronecker}
Let $F(n)$ and $F_1(n)$ count equivalence classes $\left[(1,2\!\cdot\!
  0, 1)\right]$ and $[(2, 2\!\cdot\! 1, 2)]$, and classes derived from
them, as $1/2$ and $1/3$, respectively. Define $E(0) = 1/12$, $E(4n) =
E(n)$ for $n \neq 0$, and $E(n) = F(n) - F_1(n)$ for $n \not \equiv 0
\pmod{4}$. Then
\begin{equation} \label{eq:kronecker}
\vartheta_3^3(q) = 12\sum\limits_{n = 0}^\infty E(n)q^n.
\end{equation}
\end{thm}
Though not obvious at first sight, the formula (\ref{eq:kronecker})
allows us to tabulate class numbers $h(\Delta)$. Recall Gau\ss's result
that $\tilde{F}(n)$ is a multiple of $\tilde{F}_1(n)$
\cite[\textsection 256]{gauss}:
\begin{equation} \label{eq:F1_and_F}
\tilde{F}_1(n) = \begin{cases}
\tilde{F}(n), & \textnormal{when $n \equiv 7 \pmod{8}$ or $n = 3$};\\
\tilde{F}(n)/3, & \textnormal{when $n \equiv 3 \pmod{8}$ and $n \neq 3$};\\
0, & \textnormal{when $n \not \equiv 3 \pmod{4}$.}
\end{cases}
\end{equation}
We may now prove Theorem \ref{thm:h_and_F}, which connects $h(\Delta)$
to $\tilde{F}(n)$, where $\Delta = -4n$ or $\Delta = -n$, depending on
the congruence class of $\Delta$ modulo 4.
\begin{thm} \label{thm:h_and_F}
For $\Delta < 0$ the following relation holds:
\begin{equation} \label{eq:h}
h(\Delta) = \begin{cases}
\tilde{F}(n), & \textnormal{when $\Delta \equiv 0, 1, 4 \pmod{8}$ or $\Delta = -3$};\\
\tilde{F}(n)/3, & \textnormal{when $\Delta \equiv 5 \pmod{8}$ and $\Delta \neq -3$},
\end{cases}
\end{equation}
where
\begin{equation} \label{eq:n_delta}
n = \begin{cases}
-\Delta / 4, & \textnormal{if $\Delta \equiv 0 \pmod{4}$;}\\
-\Delta, & \textnormal{if $\Delta \equiv 1 \pmod{4}$.}
\end{cases}
\end{equation}
\end{thm}
\begin{proof}
Consider a primitive binary quadratic form $(a, 2b, c)$ of determinant $D = -n$,
where $n$ is positive. When $D \equiv 1 \pmod{4}$ and $(a, 2b, c)$ is
even, i.e.\ $a$, $c$ are even and $b$ is odd, this form can be
transformed into a form $(a/2,b,c/2)$ with discriminant $\Delta \equiv
1 \pmod{4}$. This map is bijective, since every form $(a, b, c)$ of
discriminant $\Delta$ with odd $b$ corresponds to a primitive even
form $(2a, 2b, 2c)$ of determinant $D$. We conclude that $h(\Delta) =
\tilde{F}_1(n)$. When $D \not\equiv 1 \pmod{4}$, there are no
primitive even forms, and a primitive uneven form $(a, 2b, c)$ with
determinant $D$ already has a fundamental discriminant $\Delta =
4D$. There are no other forms $(a, b, c)$ of discriminant $\Delta$
with $\gcd(a, b, c) = 1$, besides those of determinant $D$, so
$h(\Delta) = \tilde{F}(n)$. In the end, we obtain the relation
(\ref{eq:h}).
\end{proof}

According to Theorems \ref{thm:kronecker} and \ref{thm:h_and_F}, by
cubing $\vartheta_3(q)$ we can tabulate $h(\Delta)$ for every
fundamental discriminant $\Delta$, except for $\Delta \equiv 1
\pmod{8}$, because in this case we have $F(n) = F_1(n)$ and thus $E(n)
= 0.$

Before proceeding further, recall the definition of a Hurwitz class number $H(n)$.
\begin{defin} \label{defin:hurwitz}
Let
\begin{equation} \label{eq:h_omega}
h_\omega(\Delta) = \begin{cases}
h(\Delta), & \textnormal{if $\Delta < -4$;}\\
1/2, & \textnormal{if $\Delta = -4$;}\\
1/3, & \textnormal{if $\Delta = -3$,}\\
\end{cases}
\end{equation}
and consider negative $\Delta = f^2 \Delta_1$, where $\Delta_1$ is a
fundamental discriminant. Then
$$
H\left(|\Delta|\right) = \sum\limits_{t\,|\,f}h_\omega\left(\frac{\Delta}{t^2}\right),
$$
is called the \emph{Hurwitz class number}.
\end{defin}

In Theorem \ref{thm:h_and_F} we determined the connection that exists
between $h(\Delta)$ and the number of \emph{primitive} uneven classes
$\tilde{F}(n)$. However, the formula (\ref{eq:kronecker}) has $F(n)$
instead of $\tilde{F}(n)$, which also take the \emph{derived} uneven
classes into account. In fact, it is not hard to prove that $F(n) =
\tilde{F}(n)$ and $F_1(n) = \tilde{F}_1(n)$ hold if and only if $n$ is
square-free. In order to establish this connection for an arbitrary
$n$, we aim to prove Theorem \ref{thm:E_and_H}, which relates $F(n)$
to the Hurwitz class number $H(n)$ or $H(4n)$, depending on the
congruence class of $n$ modulo 4. To the best of our knowledge, the
formula (\ref{eq:E_and_H}) is not present in any literature available,
though its statement for the special case of square-free $n > 4$ is well known
and can be found, for example, in the monograph of Grosswald \cite[Chapter~4, Theorem~2]{grosswald}.

\begin{thm} \label{thm:E_and_H}
Let $E(n)$ be as in Theorem \ref{thm:kronecker}. Then
\begin{equation} \label{eq:E_and_H}
E(n) = \begin{cases}
1/12, & \textnormal{when $n = 0$;}\\
E(n/4), & \textnormal{when $n \equiv 0 \pmod{4}$ and $n \neq 0$};\\
H(4n), & \textnormal{when $n \equiv 1, 2 \pmod{4}$;}\\
2H(n), & \textnormal{when $n \equiv 3 \pmod{8}$;}\\
0, & \textnormal{when $n \equiv 7 \pmod{8}$},
\end{cases}
\end{equation}
where $H(n)$ denotes the Hurwitz class number.
\end{thm}
\begin{proof}
Consider the following two cases, corresponding to square-free values of $n$:
\begin{enumerate}
\item Let $n \equiv 1, 2 \pmod{4}$. For $n = 1$, we can verify that
  the formula (\ref{eq:E_and_H}) gives us the correct result. For $n
  \neq 1$, from (\ref{eq:F1_and_F}) we know that $\tilde{F}_1(n) = 0$,
  and from (\ref{eq:h}) we know that $h(\Delta) = \tilde{F}(n)$, where
  $\Delta = -4n$ according to the relation
  (\ref{eq:n_delta}). Therefore, $E(n) = \tilde{F}(n) - \tilde{F}_1(n)
  = h_\omega(-4n)$;

\item Let $n \equiv 3 \pmod{8}$. For $n = 3$, we can verify that the
  formula (\ref{eq:E_and_H}) gives us the correct result. For $n \neq
  3$, from (\ref{eq:F1_and_F}) we know that $\tilde{F}(n) =
  3\tilde{F}_1(n)$, and from (\ref{eq:h}) we know that $h(\Delta) =
  \tilde{F}_1(n)$, where $\Delta = -n$ according to the relation
  (\ref{eq:n_delta}). We obtain $E(n) = \tilde{F}(n)-\tilde{F}_1(n) =
  2h_\omega(-n)$.
\end{enumerate}

Now, consider an arbitrary $n = g^2e$, where $e$ is square-free. If we
now recall formulas from (\ref{eq:F_and_F_tilde}), then by
Definition~\ref{defin:hurwitz} we obtain formulas for $n \equiv 1, 2$ (mod 4) and
$n \equiv 3$ (mod 8):
\begin{align*}
E(n) &= \sum\limits_{t\,|\,g}\left[\tilde{F}\left(\frac{n}{t^2}\right)
  - \tilde{F}_1\left(\frac{n}{t^2}\right)\right] \\
&= \begin{cases}
\sum\limits_{t\,|\,g}h_\omega\left(\frac{-4n}{t^2}\right) = H(4n), &
\textnormal{if $n \equiv 1, 2 \pmod{4}$;}\\
2\sum\limits_{t\,|\,g}h_\omega\left(\frac{-n}{t^2}\right) = 2H(n), &
\textnormal{if $n \equiv 3 \pmod{8}$. \qedhere} \\
\end{cases}
\end{align*}
\end{proof}

According to Theorem~\ref{thm:E_and_H}, by cubing $\vartheta_3(q)$ we
can tabulate Hurwitz class numbers $H(n)$. However, the formula
(\ref{eq:kronecker}) is quite inefficient for our purposes, as we are
interested in only fundamental discriminants $\Delta$, which
correspond to coefficients of $\vartheta_3^3(q)$ of the form $16k +
4$, $16k + 8$, $8k + 3$ and $8k + 7$ for some non-negative integer
$k$.  We expect to have around $(3/\pi^2)N \approx 0.304 N$
fundamental discriminants, satisfying $|\Delta| < N$ \cite[Section
  5.10]{cohen3}. Fortunately, there exist three alternative formulas,
namely (1.13), (1.12) and (1.14) of \cite{watson}, which can be
derived easily from (\ref{eq:kronecker}) \cite{bell}:
\begin{equation}  \label{eq:2mod4}
\sum_{k = 0}^\infty F(4k+2)q^{k} = \nabla^2(q^2)\vartheta_3(q);
\end{equation}
\begin{equation}  \label{eq:1mod4}
2\sum_{k = 0}^\infty F(4k+1)q^k =  \nabla(q^2)\vartheta_3^2(q);
\end{equation}
\begin{equation} \label{eq:3mod8}
\sum_{k = 0}^\infty F(8k+3)q^{k} = \nabla^3(q),
\end{equation}
where
\begin{align*}
\nabla(q) &= \frac{1}{2}\vartheta_2(\sqrt q)q^{-\frac{1}{8}} \\
&= \frac{1}{2}\cdot 2\sum_{k = 0}^\infty {\sqrt q}^{\left(k +
  \frac{1}{2}\right)^2} \cdot q^{-\frac{1}{8}} \\
&= \sum_{k = 0}^\infty q^{\frac{k(k+1)}{2}} = 1 + q + q^3 + q^6 + q^{10} + \ldots \,\, .
\end{align*}
In order to tabulate all class numbers corresponding to fundamental
$\Delta \not \equiv 1 \pmod{8}$ and $|\Delta| < N$, it is sufficient
to compute (\ref{eq:2mod4}) and (\ref{eq:1mod4}) to degrees $\lfloor N
/ 16 \rfloor$, and (\ref{eq:3mod8}) to degree $\lfloor N / 8
\rfloor$. This can be done by multiplying polynomials, obtained by
truncating series on the right sides of the equations above to a
specific degree.

Although this idea of reducing the class number tabulation problem
sounds good in theory, there are significant practical obstacles when
large bounds on the discriminant are considered.  In particular, the
polynomials involved are too large to fit into computer memory, so we
have to perform our multiplication out-of-core, i.e.\ with the usage
of the hard disk. We discuss the out-of-core polynomial multiplication
technique in Section \ref{sec:out-of-core}.

\section{Out-of-Core Multiplication} \label{sec:out-of-core}

In order to compute $h(\Delta)$ for all fundamental discriminants
$\Delta < 0$, satisfying $|\Delta| < N$ and $\Delta \not \equiv 1
\pmod{8}$, we aim to compute relations (\ref{eq:2mod4}),
(\ref{eq:1mod4}) and (\ref{eq:3mod8}) to degrees $\lfloor N / 16
\rfloor$, $\lfloor N / 16 \rfloor$ and $\lfloor N / 8 \rfloor$,
respectively.

In our computations, $N$ was chosen to be $2^{40}$. If we assume that
each coefficient of some polynomial $f(x)$ of degree $\lfloor N / 8
\rfloor$ fits into 4 bytes, then we would require 512 GB to fit $f(x)$
into memory. Hence, in order to store two polynomials, $f(x)$ and
$g(x)$, as well as the resulting polynomial $h(x)$, we need 1.5 TB,
not to mention that the Fast-Fourier Transform (FFT), which we use to
multiply polynomials, requires a lot of memory for intermediate
results. Such an intensive memory requirement forces us to perform
polynomial multiplication \emph{out-of-core}, i.e., with the usage of
the hard disk.

The first step is to reduce the degree of the polynomials to be
multiplied.  Following Hart et al.\ \cite{hart}, we convert
polynomials of large degree with small coefficients into polynomials
of small degree with large coefficients by utilizing \emph{Kronecker
  substitution}. Consider the polynomial
$$
f(x) = f_0  + f_1x + f_2x^2 + \ldots + f_{N - 1}x^{N-1} \in \mathbb{Z}[x]
$$
of degree $N - 1$. Fix a \emph{bundling parameter} $B$, dividing $N$,
and let $N_0 = N/B$. Then we can write $\hat f(x, y)$, satisfying
$\hat f(x^B, x) = f(x)$, as follows:
$$
\hat f(x, y) = \sum_{n = 0}^{N_0 - 1} F_n(y)x^{n} = F_0(y) + F_1(y)x +
F_2(y)x^{2} + \ldots + F_{N_0 - 1}(y)x^{N_0 - 1}, 
$$
where
$$
F_n(y) = f_{nB} + f_{nB + 1}y + \ldots + f_{(n+1)B - 1}y^{B - 1}.
$$
If all the coefficients of $f(x)$ fit into $s$ bits, we can
\emph{bundle} them by evaluating each $F_n(y)$ at $2^s$, and obtain
the following \emph{bundled polynomial} $F(x)$:
\begin{equation} \label{eq:bundled_poly}
F(x) = \sum_{n = 0}^{N_0 - 1} F_n\left(2^s\right)x^n.
\end{equation}
While $f(x)$ has coefficients of size $s$ bits and degree $N - 1$, the
bundled polynomial $F(x)$ has coefficients of size $Bs$ bits and a
smaller degree $N_0 - 1$. Now, in order to perform
a multiplication $h(x) = f(x) \times g(x)$, one has to bundle
coefficients of $g(x)$ with the same parameters $B$ and $s$, and
obtain a bundled polynomial $G(x)$. The polynomial $H(x) = F(x) \times
G(x)$, the coefficients of which fit into $(2B-1)s$ bits, will
therefore embed information on coefficients of $h(x)$.

As a technical point, note that $H(x)$ is \emph{not} a bundled
polynomial of $h(x)$.  In order to extract the coefficients of $h(x) =
\sum_{k = 0}^{N-1}h_kx^k$ from $H(x) = \sum_{n = 0}^{N_0 - 1}H_nx^n$,
a simple computation reveals that the summands of $h_k = \sum_{i =
  0}^k f_ig_{k-i}$ with $nB \leq k \leq nB + B - 2$ occur in both
$H_{n-1}$ and $H_n$ for some positive integer $n$. In fact, if we let
$H_n = \sum_{j = 0}^{2B-2}H_n^{(j)}2^{js}$, where $H_n^{(j)}$ are all
positive, then $h_k = H_{n}^{(k)} + H_{n-1}^{(B + k)}$. The only
exceptions correspond to $h_k = H_0^{(k)}$ for $k < B$, and for
$h_{tB-1} = H_{t-1}^{(tB-1)}$ for some integer $t > 1$.  Nevertheless,
it is a simple matter to recover the $h_k$ given $H(x).$

At this point we have reduced the problem to a multiplication of
smaller-degree polynomials, but with much larger coefficients.  The
next step is to reduce the coefficient sizes to the point that the
polynomials involved can be fit into available memory.  This is
accomplished via many \emph{Number Theoretic Transforms} (NTT) with
\emph{Chinese Remainder Theorem} (CRT) reconstitution\footnote{The
  paper of Hart contains a good survey on various out-of-core FFT
  methods and their applications \cite[Section 3]{hart}.}. The idea is
simple: in order to multiply two bundled polynomials $F(x)$ and $G(x)$
with large coefficients, one chooses $n$ many primes $p_0, \ldots,
p_{n-1}$, and performs reduction of coefficients of $F(x)$ and $G(x)$
modulo each $p_i$ for $0 \leq i < n$ using a \emph{remainder tree}
\cite{borodin}.  After that, $n$ pairs of polynomials are multiplied
(possibly in parallel) over each finite field $\mathbb F_{p_i}$, and
as a result, each polynomial will contain residues of $H(x) = F(x)
\times G(x)$ modulo $p_i$. In the end, the coefficients of $H(x)$ can
be reconstructed with the Chinese Remainder Theorem, and this
procedure can also be easily parallelized. Note that the intermediate
results, namely reduced polynomials and the result of polynomial
multiplications, are stored in $m$ files on the hard disk. We observed
that the choice of the number of files does not affect the performance
of our program, and suggest to set $m$ to be an integer multiple of
number of threads used for computations. In our computations, we used
64 threads and produced $m = 2^{12} = 4096$ files for each congruence
class of $\Delta.$  With this choice of $m$, each file contains at
most 10.3 million class groups, providing a reasonable balance
between file size and total number of files.

In order for the technique described previously to work one has to
know ahead an upper bound $C$ on coefficients of $h(x)$, and choose
primes such that $C < \prod_{i = 0}^{n-1}p_i$.
Depending on the amount of memory available, each $p_i$ is chosen in
such a way that the reduced polynomials in $\mathbb{F}_{p_i}[x]$ can
be comfortably multiplied in main memory.

\subsection{Computational Parameters} \label{subsec:parameters}

The choices of a bundling parameter $B$ and number of CRT primes can
be optimized based on the amount of computer memory available. In
order to make a proper choice of the \emph{bit size parameter} $s$, we
need to know how many bits are required to represent coefficients of
(\ref{eq:2mod4}), (\ref{eq:1mod4}) and (\ref{eq:3mod8}). In other
words, we need to determine an explicit, unconditional upper bound on
$H(n)$. To this end, consider the analytic class number formula
$$
h_\omega(\Delta) = \frac{1}{\pi}\sqrt{|\Delta|}L(1,\chi_\Delta), \,\,\, \textnormal{where} \,\,\, L(1,\chi_\Delta) = \sum_{m = 1}^\infty \frac{1}{m}\left(\frac{\Delta}{m}\right),
$$
$h_\omega(\Delta)$ is defined in (\ref{eq:h_omega}), and $\left(\frac{\Delta}{m}\right)$ is the Dirichlet symbol. To find the bit size of $h(\Delta)$, we utilize Ramar\'e's unconditional bound on $L(1, \chi_\Delta)$ \cite{ramare}:
\begin{equation} \label{eq:bound}
L(1,\chi_\Delta) \leq a \log |\Delta| + b,
\end{equation}
where
$$
\begin{array}{l l}
a = \frac{1}{4}, b = \frac{5}{4} - \frac{\log 3}{2}, & \textnormal{if
  $\Delta \equiv 0 \pmod{4}$};\\
a = \frac{1}{2}, b = \frac{5}{2} - \log{6}, & \textnormal{if $\Delta
  \equiv 1 \pmod{4}$}.
\end{array}
$$
We may now apply Ramar\'e's bounds (\ref{eq:bound}) to determine an
upper bound on $H(n)$ for every $n < N$ of the form $n = g^2e$, where $e$ is square-free:
$$
H(n) \leq \frac{1}{\pi}\sum_{t\,|\,g}\sqrt{\frac{n}{t^2}}\left(a\log\frac{n}{t^2} + b\right) < \frac{1}{\pi}\sqrt{n}(a\log|\Delta| + b)\sum_{t\,|\,g}\frac{1}{t}.
$$
To estimate the sum $\sum_{t\,|\,g}\frac{1}{t}$ for $N = 2^{40}$, we
picked the largest possible $g = 605395$, and found the integer $n =
554400$ that does not exceed $g$, which has the largest value of
$\sum_{t\,|\,n}\frac{1}{t} = \frac{1209}{275}$.  Then, for
\begin{equation} \label{eq:CN}
C_N = \left\lfloor \frac{1209}{275}\cdot \frac{1}{\pi}\sqrt{N}\left(a\log N + b\right) \right\rfloor, \,\,\, \textnormal{where}\,\,\, N \leq 2^{40},
\end{equation}
and $a$ and $b$ as in (\ref{eq:bound}), we have that $H(n) < C_N$ for all $n < N$.

Now we can explain how to compute the bit size parameter $s$. Recall
that the main class number tabulation formulas (\ref{eq:2mod4}),
(\ref{eq:1mod4}) and (\ref{eq:3mod8}) require $C_N$, $2C_N$ and $3C_N$
as their upper bounds, respectively. Considering this, the formula for
$s$ is given as
\begin{equation} \label{eq:s}
s = \begin{cases}
\lceil\log_2C_N\rceil, & \textnormal{for (\ref{eq:2mod4})};\\
\lceil\log_2(2C_N)\rceil, & \textnormal{for (\ref{eq:1mod4});} \,\,\,\,\,\, \textnormal{where $N \leq 2^{40}$.}\\
\lceil\log_2(3C_N)\rceil, & \textnormal{for (\ref{eq:3mod8}),}
\end{cases}
\end{equation}

Finally, we need to determine how many primes to choose with respect
to the bundling parameter $B$ in order to restore coefficients of
$H(x) = F(x) \times G(x)$, which all fit into $s$ bits. Recall that
each coefficient of $H(x)$ has size $(2B - 1)s$. In order to restore
coefficients of $H(x)$ with the CRT algorithm, we need to pick the
primes $p_0, \ldots, p_{n-1}$ so that $(2B-1)s < \log_2(p_0 \cdot
\ldots \cdot p_{n-1})$. In our implementation, we chose the smallest
prime $p_0$ exceeding some positive lower bound $P$, and $n - 1$
primes $p_1, \ldots, p_{n-1}$, which consecutively follow after
$p_0$. We choose $n$ so that
$$
(2B-1)s \leq n\log_2p_0 < \sum_{i = 0}^{n-1} \log_2p_i,
$$
i.e.\
\begin{equation} \label{eq:n}
n = \left\lceil\frac{(2B - 1)s}{\log_2 p_0}\right\rceil.
\end{equation}
Note that for large $p_0$ and small $n$, the difference between
$n\log_2p_0$ and $\sum_{i = 0}^{n-1} \log_2p_i$ becomes
negligible. Following Hart et al. \cite{hart}, we chose $p_0$ to be
the smallest prime exceeding $P = 2^{62}$, which fits into a single
machine word on a 64-bit system.

\subsection{Complexity Analysis} \label{subsec:complexity}

Before proceeding to the complexity analysis, we first summarize the
process of computation of $h(x) = f(x) \times g(x)$. Given two
polynomials, $f(x)$ and $g(x)$, both of degree $N - 1$, the bundling
parameter $B$ (which for convenience divides $N$), the bit size
parameter $s$, and $n$ primes $p_0, \ldots, p_{n-1}$, we compute the
product of two polynomials in five stages:
\begin{enumerate}
  \item Compute the bundled polynomials $F(x)$ and $G(x)$ of $f(x)$
    and $g(x)$, respectively, using Kronecker substitution;

  \item Reduce the coefficients of $F(x)$ and $G(x)$ modulo primes
    $p_0, \ldots, p_{n-1}$ using the remainder tree \cite{borodin} in
    order to obtain the reduced polynomials $F_{p_i}(x)$ and
    $G_{p_i}(x)$ in $\mathbb F_{p_i}[x]$ for $0 \leq i < n$;

  \item Compute $H_{p_i}(x) = F_{p_i}(x) \times G_{p_i}(x)$ in
    $\mathbb F_{p_i}[x]$ for each $0 \leq i < n$;

  \item Compute $H(x)$ (which is equal to $F(x) \times G(x)$) by
    reconstructing its coefficients from $H_{p_0}(x), \ldots,
    H_{p_{n-1}}(x)$ with the CRT algorithm;

  \item Extract the coefficients of $h(x) = f(x) \times g(x)$ from
    $H(x)$.
\end{enumerate}

The pseudocode of this algorithm can be found in the original paper of
Hart et al.\ \cite[Section 4.1]{hart}. Note that their algorithm
corresponds to the case $s = 16$. The generalized version of the
algorithm for an arbitrary positive integer $s$ can be found in
\cite[Section 4.2]{mosunov1}. In Theorem \ref{thm:complexity1}, we
give the asymptotic bit-complexity of this algorithm as a function of
the polynomial degree ($N$) and the bundling and bit size parameters.
\begin{thm} \label{thm:complexity1}
Consider two polynomials, $f(x)$ and $g(x)$, both of degree $N - 1$,
whose coefficients can be initialized in $O(N)$ bit operations. Using
the technique described above, the product $h(x) = f(x) \times g(x)$
can be computed in
\begin{equation} \label{eq:o1}
O\left(Ns\left(\log(Bs)\right)^{2+\varepsilon} + Ns\left(\log \frac{N}{B}\right)^{1+\varepsilon}\right)
\end{equation}
bit operations, where $B$ is the bundling parameter, and $s$ is the
bit size parameter.
\end{thm}
\begin{proof}
We analyze each of the five stages of the algorithm.  The computation
of bundled polynomials $F(x)$ and $G(x)$ in stage (1) consists of
sequential applications of logical shifts and ORs, and requires $O(N)$
bit operations. Each bundled polynomial has $N / B$ coefficients, so
the multimodular reduction phase (2) requires $N/B$ reductions modulo
$n$ primes $p_0, \ldots, p_{n-1}$. We use a \emph{remainder tree} to
reduce each coefficient $C$ of a bundled polynomial modulo $p_0,
\ldots, p_{n-1}$. This technique allows us to compute $C \bmod p_0, C
\bmod p_1, \ldots, C \bmod p_{n-1}$ in $O\left(t\left(\log
t\right)^{2+\varepsilon}\right)$ bit operations, where $t$ is the
total number of bits in $C, p_0, \ldots, p_{n-1}$ \cite[Section
  3]{borodin}. Since each coefficient of a bundled polynomial fits
into $Bs$ bits, we conclude that the multimodular reduction phase requires
$$
O\left(\frac{N}{B}\cdot t\left(\log t\right)^{2+\varepsilon}\right) = O\left(\frac{N}{B} \cdot Bs\left(\log (Bs)\right)^{2+\varepsilon}\right) = O\left(Ns\left(\log(Bs)\right)^{2+\varepsilon}\right).
$$
bit operations.

In stage (3), the multiplication of $n$ pairs of polynomials of degree
$N/B - 1$ is performed with the Sch\"onhage-Strassen algorithm
\cite[Sections 8.2 -- 8.4]{gathen}. This algorithm requires $O(N\log N
\log\log N)$ bit operations to multiply two polynomials of degree $N$.
Hence, stage (3) requires
$$
O\left(n\frac{N}{B}\log\frac{N}{B}\log\log\frac{N}{B}\right) = O\left(Bs\frac{N}{B}\log\frac{N}{B}\log\log\frac{N}{B}\right) = O\left(Ns\left(\log\frac{N}{B}\right)^{1+\varepsilon}\right)
$$
bit operations.

Finally, consider stages (4) and (5), i.e., the CRT reconstitution and
extraction of coefficients. Though the latter involves certain
sophisticated techniques, it simply iterates over all $N$ coefficients
of the resulting polynomial $H(x)$, and therefore requires $O(N)$ bit
operations. Now, consider the CRT reconstitution in stage (4). For the
CRT, we use the divide-and-conquer technique
\cite[Section~4]{hart}. For $n_1$ integer coefficients of size $n_2$
bits, this approach allows us to complete the restoration of a
coefficient in
$O\left(n_2\left(\log(n_1n_2)\right)^{2+\varepsilon}\right)$ bit
operations. In our case, $n_2$ is constant and $n_1 = n$, where $n$ is
the number of primes in use. In total, there are $N/B$ coefficients to
restore, which means that the number of bit operations required is in
$$
O\left(\frac{N}{B}(\log n)^{2+\varepsilon}\right) =
O\left(\frac{N}{B}\left(\log (Bs)\right)^{2+\varepsilon}\right).
$$

Since $s > \log^{-1}B$, the asymptotic running time of the
initialization phase dominates the running time for the CRT
reconstitution phase. Combining the costs for the initialization and
multiplication phases yields the result (\ref{eq:o1}).
\end{proof}

Note that the class number tabulation formulas (\ref{eq:2mod4}),
(\ref{eq:1mod4}) and (\ref{eq:3mod8}) require \emph{two} polynomial
multiplications. For example, in order to determine (\ref{eq:3mod8}),
we first have to perform the multiplication $\nabla^2(q) = \nabla(q)
\times \nabla(q)$, followed by the computation of $\nabla^3(q) =
\nabla^2(q) \times \nabla(q)$. In practice, we use a different
approach; that is, we initialize $\vartheta_3^2(q)$ (or $\nabla^2(q)$,
or $\nabla^2(q^2)$) \emph{directly}, which allows us to evaluate the
formula using one polynomial multiplication instead of two.

We describe the initialization mechanism for the example of
$\vartheta_3(q)$. A similar approach can be used to initialize
$\nabla(q)$ and $\nabla(q^2)$. We compute the first $N$ coefficients
of $\vartheta_3(q)$ block by block, using a certain \emph{partition
  size} $S$ dividing $N$. The initialization algorithm for the block of
$S$ coefficients from $M$ to $M + S - 1$ requires $O(\sqrt{M+S})$ bit
operations, as there are precisely $\left\lfloor \sqrt{M+S} -
\sqrt{S}\right\rfloor$ perfect squares between $M$ and $M+S-1$; we can
easily iterate over all of them within a single loop. Summing over
$N/S$ blocks, we obtain that the initialization of $N$ coefficients of
$\vartheta_3(q)$ requires $O(\sqrt S) + O(\sqrt{2S}) + \ldots +
O(\sqrt N) = O(N\sqrt N / S)$ bit operations. In order to achieve a
linear time for initialization, we choose $S = O(\sqrt N)$.

In turn, the initialization of a block of coefficients of
$\vartheta_3^2(q)$ requires two nested loops, which result in $O(S)$
bit operations. Summing $O(S) + O(2S) + \ldots + O(N) = O(N^2 / S)$,
we conclude that $S$ has to grow proportionally to $N/(\log N)^k$ for
some non-negative integer $k$ in order for $\vartheta_3^2(q)$ to be
initialized in linear or pseudo-linear time. Of course, this is
unreasonable. However, for small $N,$ initializing $\vartheta_3^2(q)$
directly works well in practice, even though it is worse
asymptotically than using two sequential polynomial
multiplications. 

We now state the asymptotic complexity of the complete class
number tabulation method without including the initialization costs of
$\vartheta_3^2(q)$, $\nabla^2(q)$ or $\nabla^2(q^2)$ mentioned
above. We obtain Corollary~\ref{cor:complexity2} by applying the
formula for $s$ in (\ref{eq:s}) to Theorem~\ref{thm:complexity1}.
\begin{cor} \label{cor:complexity2}
The class number tabulation algorithm requires
\begin{equation} \label{eq:o2}
O\left(N\log N (\log B)^{2+\varepsilon} + N\log N\left(\log\frac{N}{B}\right)^{1 + \varepsilon}\right)
\end{equation}
bit operations.
\end{cor}

In theory, all steps of the algorithm can be parallelized trivially,
yielding a speed-up of $T$ using $T$ threads --- see
\cite[Chapter~4]{mosunov1} for a complete description and analysis.
In practice, such optimal speedup is difficult to achieve due to the
cost of managing the threads and the assumption that all disk I/O is
being done in parallel, both reading \emph{and} writing.  Special hard
disks designed for large-scale parallel applications, such as those
used in our experiments, are necessary to get the most out of
parallelization.

The method described by Ramachandran
et.\ al.\ \cite{jacobson,ramachandran} has bit complexity
$O(|\Delta|^{1/4+\varepsilon})$ for each discriminant, and thus
$O(N^{5/4+\varepsilon})$ for all $|\Delta| < N,$ including the
ERH-verification step.  Our new algorithm computes the class numbers
asymptotically faster, but we can only use it for $\Delta \not \equiv
1 \pmod{8}$ and require a more expensive method for the remaining
congruence class.  In addition, Ramachandran's method also computes
the class group structures.  We describe our approach to this part of
the problem in the next section.

\section{Unconditional Class Group Tabulation} \label{sec:clgrp}

The class number tabulation technique, described in
Sections~\ref{sec:clnum} and \ref{sec:out-of-core}, allows us to
compute unconditionally all class numbers $h(\Delta)$ with $\Delta
\not \equiv 1 \pmod{8}$ and $|\Delta| < N$. To resolve the structure
of each class group $Cl_\Delta,$ we use the algorithm due to Buchmann,
Jacobson, and Teske (BJT) \cite[Algorithm 4.1]{buchmann}, suitable for
any generic group $G$. This algorithm iteratively builds up the set of
generators $\bm{\alpha}$ of $G$, and terminates whenever the size of
the subgroup $\langle \bm{\alpha} \rangle$ generated by $\bm{\alpha}$
matches $|G|$.

Note that tabulating class numbers for $\Delta \not \equiv 1
\pmod{8}$ has another major advantage, aside from the fact that we
were able to produce the size of each $Cl_\Delta$ unconditionally and
did not require an additional verification step. Given the
factorization of $h(\Delta) = p_1^{e_1}\cdot p_2^{e_2} \cdot \ldots
\cdot p_k^{e_k}$, we can ignore those primes $p_i$ for $1 \leq i
\leq k$ which have $e_i = 1$, as it means that the $p_i$-group of
$Cl_\Delta$ is guaranteed to be cyclic. We can therefore resolve
the structure of a smaller subgroup $G$ of $Cl_\Delta$, satisfying
$$
|G| = \prod\limits_{p_i^2 | h(\Delta)} p_i^{e_i}.
$$
In practice, $|G|$ was much smaller than $h(\Delta)$ frequently, so
this method worked very well.

For $\Delta \equiv 1 \pmod{8}$, where our tabulation method does not
produce any class numbers, we used the same method as in
\cite{jacobson}.  The Buchmann-Jacobson-Teske algorithm can still be
used to compute class groups without knowing the class numbers a
priori.  In this case, it is sufficient to provide a lower bound $h^*$
such that $h^* \leq h(\Delta) \leq 2h^*$ in order to be certain that
the whole group was generated --- once the size of the subgroup
$\langle \bm{\alpha} \rangle$ generated by $\bm{\alpha}$ exceeds
$h^*,$ we know that we have the entire group.

As described in \cite{jacobson}, the main issue with this approach is
that the best method to determine the lower bound $h^*$ requires the
ERH-dependent averaging method of Bach \cite{bach1}.  To eliminate the
ERH dependency, we again followed \cite{jacobson} and applied the
Eichler-Selberg trace formula. This formula gives an expression for
the trace of the Hecke operator $T_n$ acting on the space of cusp
forms $S_k(\Gamma_0(N),\chi).$ When applied to the case where $k=2,$
$N=1,$ and $\chi$ the trivial character, the trace formula reduces to
the following equality involving Hurwitz class numbers:
\begin{equation} \label{eq:verify}
H(4n) + 2\sum_{t = 1}^{\lfloor \sqrt{4n} \rfloor}H(4n-t^2) = 2\left(\sum_{\substack{d\,|\,n\\d \geq \sqrt n}} d\right) - \sigma(n)\sqrt n + \frac{1}{6}\chi(n),
\end{equation}
where $\sigma(n)$ is the indicator function, which is 1 whenever $n$
is a perfect square and 0 otherwise \cite[Formula 2.2]{jacobson}. Due to the
nature of the BJT algorithm, the size of the class group computed will
always divide $h(\Delta)$. Therefore, if one or more of our computed
class numbers are wrong, then (\ref{eq:verify}) will detect this
because the left hand side will be less than the right hand side.
Note that in our case the only class numbers $h(\Delta)$ that require
verification are those with $\Delta \equiv 1 \pmod{8}$, so it is
sufficient in our case to verify that the equality (\ref{eq:verify})
holds only for \emph{even} values of $n$.

One method to use (\ref{eq:verify}) to verify all $h(\Delta_1)$ with
$\Delta_1$ fundamental and $|\Delta_1| < N,$ as suggested in
\cite{jacobson}, is to first compute the smallest set of $n$ values
such that every fundamental discriminant $\Delta_1$ divides at least
one Hurwitz class number in the formula.  However, a more efficient
approach was later suggested by Ramachandran \cite{ramachandran},
based on simplifying the computation of (\ref{eq:verify}) for all
values of $n$ between $1$ and $\lfloor N / 4 \rfloor.$ 

Following Ramachandran \cite[Formulas~4.10, 4.11]{ramachandran}, but
adjusting for the fact that we only need to verify discriminants
congruent to $1 \bmod 8,$ we define two quantities, $LHS$ and $RHS$,
as follows:
\begin{equation} \label{eq:lhs}
LHS = \left(\sum_{\substack{\Delta \equiv 0 \pmod 8\\ |\Delta| \leq 8X}}H(|\Delta|)\right) + 2\left(\sum_{\substack{\Delta \equiv 0, 1 \pmod 4\\|\Delta| \leq 8X}}r(\Delta,X)H(|\Delta|)\right);
\end{equation}
\begin{equation} \label{eq:rhs}
RHS = \sum_{n = 1}^X\left(2\left(\sum_{\substack{d\,|\,2n\\d \geq \sqrt{2n}}} d\right) - \chi(2n)\sqrt{2n} + \frac{1}{6}\chi(2n)\right).
\end{equation}
Here, $r(\Delta, X)$ counts the number of solutions to the equation
$\Delta = t^2 - 8n$ for $1 \leq n \leq X$:
\begin{equation} \label{eq:r}
r(\Delta, X) = \begin{cases}
0, & \textnormal{if $\Delta \equiv 5$ (mod 8);}\\
\lfloor (Y+1)/2 \rfloor, & \textnormal{if $\Delta \equiv 1$ (mod 8);}\\
\lfloor (Y+2)/4 \rfloor, & \textnormal{if $\Delta \equiv 4$ (mod 8);}\\
\lfloor Y/4 \rfloor, & \textnormal{if $\Delta \equiv 0$ (mod 8)},\\
\end{cases}
\end{equation}
where $Y = \lfloor \sqrt{8X + \Delta} \rfloor$.  We computed both
$LHS$ and $RHS$ in parallel for $X = \lfloor N / 8 \rfloor$, where $N
= 2^{40}$.  The expression $LHS$ is evaluated using the table of class
numbers of fundamental discriminants computed using the BJT method;
see \cite[Algorithm~4.1]{ramachandran} for pseudocode.  Though
computationally more intensive, the calculation of the $RHS$ is more
straightforward and easily parallelizable. In order to compute the
divisors for each $n \leq X$, we use the formula (\ref{eq:rhs}) in
conjunction with a segmented sieve.

\section{Performance} \label{sec:performance}

For the class number tabulation using out-of-core polynomial
multiplication, we used the FLINT library for number theory,
maintained by Hart \cite{flint}. In particular, we used the
\texttt{nmod\_poly\_mullow} routine for polynomial multiplication in
$\mathbb F_p[x]$. The FLINT library also contains subroutines
for fast reduction modulo primes $p_0, \ldots, p_{n-1}$ and CRT
reconstitution, respectively. We used OpenMP for parallelization.

For the class group computation, we used Sayles's libraries
\texttt{optarith} and \texttt{qform}, which contain fast
implementations of binary quadratic form arithmetic
\cite{sayles1,sayles2}, including implementations targeted to
machine-size operands that avoid multi-precision integer arithmetic.
We also use Message Passing Interface (MPI) for parallelization. The
source code for our program can be found in \cite{mosunov2}.

Our computations were performed on WestGrid's supercomputer Hungabee,
located at the University of Alberta, Canada \cite{westgrid}. Hungabee
is a 16 TB shared memory system with 2048 Intel Xeon cores, 2.67GHz
each. Each user of Hungabee may request at most 8 GB of memory per
core. Also, Hungabee provides a high performance 53 TB storage space,
which allows to write to multiple disks in parallel. Note that the
fast disk I/O requirement is essential for the high performance of our
program.

We first discuss our class number tabulation program. We performed
three polynomial multiplications, described in formulas
(\ref{eq:2mod4}), (\ref{eq:1mod4}) and (\ref{eq:3mod8}). After running
several tests, we determined that Hungabee can comfortably multiply
polynomials of $2^{25}$ coefficients without requiring additional
memory. This observation allowed us to make the proper choice of a
bundling parameter $B$. 

Table \ref{tab:parameters} contains the list of parameters which we
used for our computations, and the amount of disk space needed to
store intermediate computations required for the polynomial
multiplication.
\begin{table}[htb]
\centering
\caption{Computational parameters}
\label{tab:parameters}
\begin{tabular}{| c | c || c | c | c | c | c | c |}
\hline
Formula & $\Delta$ & $N$ & $B$ & $C$ & $s$ & $n$ & Disk space\\
\hline
\hline
$\nabla^2(q^2)\cdot \vartheta_3(q)$ & 8 (mod 16) & $2^{36}$ & $2^{11}$ & 11199314 & 24 & 1586 & 859 GB\\
\hline
$\vartheta_3^2(q) \cdot \nabla(q^2)$ & 12 (mod 16) & $2^{36}$ & $2^{11}$ & 11199314 & 25 & 1652 & 893.4 GB\\
\hline
$\nabla^2(q) \cdot \nabla(q)$ & 5 (mod 8) & $2^{37}$ & $2^{12}$ & 21381515 & 26 & 3435 & 1855 GB\\
\hline
\end{tabular}
\end{table}
Here, $C$ is the bound on $H(|\Delta|)$ defined in (\ref{eq:CN}), $s$
is the bit size parameter (\ref{eq:s}), and $n$ is the number of
63-bit primes required for correct CRT reconstitution
(\ref{eq:n}). For each multiplication, we requested 64 processors and
8 GB of memory per core. The number of files was chosen to be $m =
4096$.  Table \ref{tab:clnum_timings} lists timings for each of the
three class number tabulation algorithms.
\begin{table}[htb]
\centering
\small
\caption{Timings for the class number tabulation program}
\label{tab:clnum_timings}
\begin{tabular}{| c | c || c | c |}
\hline
$f(x)$ & $g(x)$ & CPU time & Real time\\
\hline
\hline
$\nabla^2(q^2)$ & $\vartheta_3(q)$ & 23d 11h 10m 56s & 8h 47m 59s\\
\hline
$\vartheta_3^2(q)$ & $\nabla(q^2)$ & 29d 21h 2m 56s & 11h 12m 14s\\
\hline
$\nabla^2(q)$ & $\nabla(q)$ & 68d 5h 8m 16s & 25h 34m 49s\\
\hline
\end{tabular}
\end{table}

Table~\ref{tab:timings_class_group} contains timings for computing the
class group structures.
As expected, for $\Delta \equiv 1 \pmod{8}$ our program takes
significantly more time, since Ramachandran's approach requires the
computation of the whole group. If we assume that all $\Delta$ were
handled using solely Ramachandran's technique, then 64 processors
would complete the (conditional) tabulation to $2^{40}$ in 80d 11h 9m
48s, as opposed to 31d 22h 45m 8s (counting the class number
tabulation and the verification). Note that 81.13\% of time in our
computations was spent on the computation of $Cl_\Delta$ for $\Delta
\equiv 1 \pmod{8}$ and the verification of the result.

\begin{table}[!ht]
\centering
\caption{Timings for the class group tabulation program}
\label{tab:timings_class_group}
\begin{tabular}{| c || c | c | c |}
\hline
$\Delta$ & CPU time & Real time & \# processors\\
\hline
$\Delta \not \equiv 1$ (mod 8) & 267d 4h 31m 40s & 4d 3h 26m 44s & 64\\
\hline
$\Delta \equiv 1$ (mod 8) & 1657d 22h 12m 6s & 39h 28m 27s & 1008\\
\hline
\end{tabular}
\end{table}

We also observe that the structures of $Cl_\Delta$ for all congruence
classes with $\Delta \not \equiv 1 \pmod{8}$ were computed over 6.25
times faster than those with $\Delta \equiv 1 \pmod{8}.$ If we include
the verification cost for the $1 \bmod 8$ case and the class number
tabulation for the rest, then the entire computation for all $\Delta
\not \equiv 1 \pmod{8}$ is roughly 4.72 times faster than that for $1
\bmod 8.$ Such a significant speedup occurs due to the fact that in
57.34\% of the cases $h(\Delta)$ had a square-free factor exceeding
$\sqrt{h(\Delta)}$, which means that the size of the subgroup that we
had to resolve was small relative to the size of the group
itself. Moreover, in 1.67\% of the cases $h(\Delta)$ were square-free,
which means that no resolution of class groups was needed at all.  In
general, over 85.13\% of class numbers $h(\Delta)$ possessed a
square-free part larger than 1. In Table~\ref{tab:count_h}, we present
the counts of class numbers up to $2^{40}$ with various divisibility
properties. In particular, column 3 counts class numbers with
square-free part greater than 1, column 4 counts $h(\Delta)$ with
square-free part exceeding $\sqrt{h(\Delta)}$, and column 5 counts
class numbers that are square-free. Our data is separated into four
congruence classes.
\begin{table}[h]
\centering
\caption{Counts of $h(\Delta)$ with various divisibility properties}
\begin{tabular}{| c || r | r | r | r |}
\hline
$\Delta$ & \multicolumn{1}{c|}{Total $\Delta$} & \multicolumn{1}{c|}{$p\,|\,h,\,p^2\,\nmid\,h$} & \multicolumn{1}{c|}{$h=g^2e,\,e > \sqrt{h}$} & \multicolumn{1}{c|}{square-free $h$}\\
\hline
\hline
8 (mod 16) & 55701909754 & 47077629143 & 32012088117 & 941347842\\
\hline
12 (mod 16) & 55701909855 & 47091713960 & 31927265003 & 915383075\\ 
\hline
5 (mod 8) & 111403819688 & 95517292502 & 63828635213 & 8828052571\\
\hline
1 (mod 8) & 111403819373 & 94502061670 & 55851403024 & 7295483368\\
\hline
\end{tabular} \label{tab:count_h}
\end{table}

Note that the counts for $\Delta \equiv 1 \pmod{8}$ were not included
in the percentages listed above.  The counts are similar to the other
congruence classes, but divisibility properties of the class number
played no role in the computation for the $1 \bmod 8$ case as the
class numbers were not computed first.  It should be emphasized that
the rapid computation of all class numbers using theta-series is what
allowed us to take advantage of these properties when resolving the
group structures.

Finally, we compare the performance of our program to the
implementation of Ramanchandran \cite{ramachandran}, and the
\texttt{quadclassunit0} routine of the PARI/GP library
\cite{pari}. For the class group resolution, the latter implementation
uses Hafner and McCurley's subexponential index calculus algorithm
\cite{mccurley,hafner}. For each implementation, we used a single
Intel Xeon 2.27GHz processor to compute $Cl_\Delta$ for every
fundamental $\Delta < 0$ such that $|\Delta|$ lies in the interval
from $2^{39}$ to $2^{39} + 2^{20}$.  For this computation, we used the
ERH-dependent version of the BJT algorithm for our implementation and
that of Ramachandran (i.e., no prior class number tabulation nor
verification in either case).  The resulting timings are listed in
Table~\ref{tab:comparison}.
\begin{table}[htb]
\centering
\caption{Timings for various class group tabulation implementations}
\label{tab:comparison}
\begin{tabular}{| c | c | c || c | c | c |}
\hline
$|\Delta_{min}|$ & $|\Delta_{max}|$ & Total $\Delta$ & Our program & \cite{ramachandran} & PARI/GP\\
\hline
\hline
$2^{39}$ & $2^{39} + 2^{20}$ & 318729 & 438s & 736s & 1181s\\
\hline
\end{tabular}
\end{table}

All three implementations yield correct results under the assumption
of the ERH. Though our program and Ramachandran's implementation
implementation use the same algorithm, 
it significantly outperforms the latter. We
believe that the optimized binary quadratic form arithmetic in Sayles's
libraries \texttt{optarith} and \texttt{qform} \cite{sayles1,sayles2}
used in our program accounts for the improvement.  

Note that asymptotically the Hafner-McCurley algorithm (with
subexponential complexity in $\log |\Delta|$) is superior to the BJT
algorithm (exponential complexity).  Thus, although it should be faster for a
sufficiently large discriminant bound, our results show that the bound
$2^{40}$ is still below the crossover point.

\section{Numerical Results} \label{sec:results}

\subsection{Bounds on $L(1,\chi_\Delta)$}
In 1928, Littlewood \cite{littlewood} demonstrated that, assuming the
ERH,
\begin{equation} \label{eq:littlewood}
(1 + o(1))(c_1\log\log|\Delta|)^{-1} < L(1,\chi_\Delta) < (1 + o(1)) c_2\log\log|\Delta|,
\end{equation}
where
$$
c_1 = \frac{12e^\gamma}{\pi^2}\,\,\,\textnormal{and}\,\,\,c_2=2e^\gamma\,\,\,\textnormal{when}\,\,\,2\,\nmid\,\Delta;
$$
$$
c_1 = \frac{8e^\gamma}{\pi^2}\,\,\,\textnormal{and}\,\,\,c_2=e^\gamma\,\,\,\textnormal{when}\,\,\,2\,|\,\Delta,
$$
and $\gamma \approx 0.57722$ is the Euler-Mascheroni constant. Later,
Shanks studied these bounds more carefully by defining two quantities,
$$
ULI = \frac{L(1,\chi_\Delta)}{c_2\log\log|\Delta|}\,\,\,\textnormal{and}\,\,\,LLI=L(1,\chi_\Delta)c_1\log\log|\Delta|,
$$
and ignoring $o(1)$ term in Littlewood's estimates
\cite{shanks}. These quantities allow us to test whether Littlewood's
bounds are violated, for if the ERH does not hold, then for large
$|\Delta|$ we might find $ULI > 1$ or $LLI < 1$. Note that there
\emph{are} small $\Delta$ such that $ULI > 1$ or $LLI < 1$, namely
$\Delta = -3, -4, -163$. We assume that values of $ULI$ and $LLI$ for
these discriminants are largely influenced by $o(1)$ terms. Aside from
$\Delta = -3, -4, -163$, we did not find any occurrences of
discriminants which violate Littlewood's bounds. Ignoring these
$\Delta$, the largest $ULI \approx 0.85183$ corresponds to $\Delta =
-8$, and the smallest $LLI \approx 1.00944$ corresponds to $\Delta =
-232$.

In addition to Littlewood's bounds, we also studied the growth of
$L(1, \chi_\Delta).$ In Table~\ref{tab:L_max}, we list successive
maximas of $L(1, \chi_\Delta)$ which did not occur in Table 5.3 of
\cite{ramachandran}. The last discriminant found was $\Delta =
-685122125399$, which corresponds to the largest $L(1, \chi_\Delta)
\approx 8.47178$ with $|\Delta| < 2^{40}$. As for successive minimas
of $L(1, \chi_\Delta)$, no new discoveries were made. The smallest
$L(1, \chi_\Delta) \approx 0.17070$ corresponds to $\Delta =
-107415709003$.
\begin{table}[htb]
\centering
\caption{Successive $L(1,\chi_\Delta)$ maxima}
\label{tab:L_max}
\begin{tabular}{| c || c | c |}
\hline
$|\Delta|$	& $L(1,\chi_\Delta)$	& $ULI$\\
\hline
\hline
210015218111	& 8.26604	& 0.71164\\
\hline
332323080311	& 8.30989	& 0.71161\\
\hline
503494619759	& 8.31253	& 0.70848\\
\hline
603231310919	& 8.32466	& 0.70807\\
\hline
685122125399	& 8.47178	& 0.71957\\
\hline
\end{tabular}
\end{table}

\subsection{The Cohen-Lenstra Heuristics}
In 1984, Cohen and Lenstra presented several powerful heuristics on
the structure of the odd part of the class group $Cl_\Delta$
\cite{cohen2}. The odd part $Cl_\Delta^*$ is simply the largest
subgroup of $Cl_\Delta$ with an \emph{odd} cardinality.
\begin{conj}[{\cite[C1]{cohen2}}]
Define
$$
\eta_k(l) = \prod\limits_{i = 1}^k\left(1 - \frac{1}{l^i}\right)\,\,\,\textnormal{and}\,\,\,C_\infty = \prod\limits_{i = 2}^\infty\zeta(i) \approx 2.294856589,
$$
where $\zeta(s)$ denotes the Riemann zeta function. For $\Delta < 0$,
the probability that the odd part of the class group $Cl_\Delta$ is
cyclic is
\begin{equation} \label{eq:pr_cyclic}
\textnormal{Pr($Cl_\Delta^*$ is cyclic)} = \frac{315\zeta(3)}{6\pi^4\eta_\infty(2)C_\infty} \approx 0.977575.
\end{equation}
\end{conj}

\begin{conj}[{\cite[C2]{cohen2}}]  
Let $l$ be an odd prime. For $\Delta < 0$, the probability that $l$ divides $h(\Delta)$ is
\begin{equation} \label{eq:l_divides_h}
\textnormal{Pr($l\,|\,h(\Delta)$)} =1 - \eta_\infty(l).
\end{equation}
\end{conj}

\begin{conj}[{\cite[C5]{cohen2}}]
Let $l$ be an odd prime. For $\Delta < 0$, the probability that the
$l$-rank of $Cl_\Delta$ is equal to $r$ is
\begin{equation} \label{eq:p-rank_is_r}
\textnormal{Pr($l$-rank = $r$)} = \frac{\eta_\infty(l)}{l^{r^2}\eta_r(l)^2}.
\end{equation}
\end{conj}

In order to study these conjectures, we follow the approach of
Jacobson et al.\ and introduce three functions: $c(x)$, $p_l(x)$ and
$p_{l,r}(x)$ \cite[Section 3.2]{jacobson}:
$$
c(x) = \frac{\textnormal{\# of $Cl_\Delta^*$ cyclic with $|\Delta| < x$}}{\textnormal{\# of $\Delta$ with $|\Delta| < x$}}\,\,\,/\,\,\,\textnormal{Pr($Cl_\Delta^*$ is cyclic)};
$$
$$
p_l(x) = \frac{\textnormal{\# of $h(\Delta)$ divisible by $l$ with $|\Delta| < x$}}{\textnormal{\# of $\Delta$ with $|\Delta| < x$}}\,\,\,/\,\,\,\textnormal{Pr($l\,|\,h(\Delta)$)};
$$
$$
p_{l,r}(x) = \frac{\textnormal{\# of $Cl_\Delta$ with $l$-rank = $r$ and $|\Delta| < x$}}{\textnormal{\# of $\Delta$ with $|\Delta| < x$}}\,\,\,/\,\,\,\textnormal{Pr($l$-rank = $r$)}.
$$

If the Cohen-Lenstra heuristics hold, we would expect each of those
functions to approach 1 as $x$ grows. We observe this behavior in
Figures \ref{fig:pl} and \ref{fig:plr}, which plot $p_l(x)$ and
$p_{l,2}(x)$ for $l = 3, 5, 7$, respectively. The values of $c(x)$, as
well as the counts of non-cyclic $Cl_\Delta^*$, are presented in
Table~\ref{tab:CL_non_cyclic}. Note that our counts differ from the
ones given in \cite[Table~3]{jacobson}. For example, the total number
of non-cyclic $Cl_\Delta^*$ for $|\Delta| < 10^{11}$ given in
\cite[Table~3]{jacobson} is $603101904$, whereas our count in
Table~\ref{tab:CL_non_cyclic} suggests that this number is
$636501087$. In general, our counts are over 1.044 times larger than
the counts given in \cite[Table~3]{jacobson}; this ratio grows with
$x$ and reaches $1.055$ for $x = 10^{11}$. We argue that values in
Table~\ref{tab:CL_non_cyclic} are correct, because the output of
algorithms for small $N < 10^{9}$ matches that of PARI/GP
\cite{pari}. Finally, in Tables~\ref{tab:CL_pl1} and \ref{tab:CL_plr}
we count the total number of $h(\Delta)$ divisible by a prime $l$, and
class groups with a certain $l$-rank.

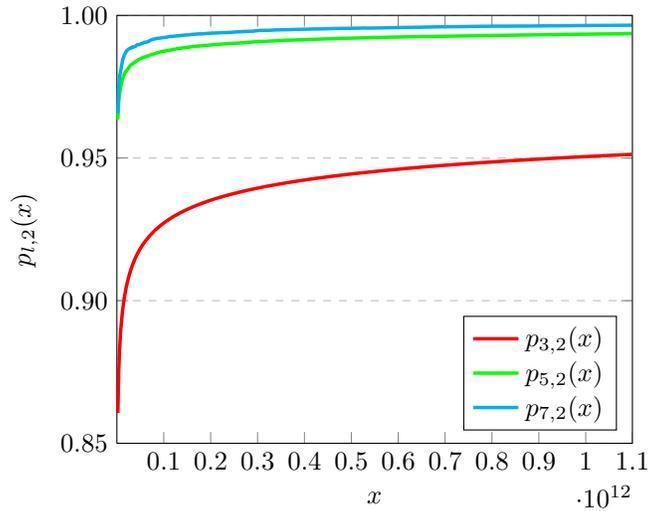
\input{plot_CL_pl.tex}

\input{plot_CL_plr.tex}

\begin{table}[htb]
\centering
\caption{Number of noncyclic odd parts of class groups}
\label{tab:CL_non_cyclic}
\begin{tabular}{| r || r | r | c | c |}
\hline
\multicolumn{1}{|c||}{$x$}	& \multicolumn{1}{c|}{Total}	& \multicolumn{1}{c|}{Non-cyclic}	& Percent	& $c(x)$\\
\hline
\hline
$10^{11}$	& 30396355052	& 636501087	& 2.09400	& 1.00152\\
\hline
$2\cdot 10^{11}$	& 60792710179	& 1283029629	& 2.11050	& 1.00135\\
\hline
$3\cdot 10^{11}$	& 91189065248	& 1932535723	& 2.11926	& 1.00126\\
\hline
$4\cdot 10^{11}$	& 121585420327	& 2583844783	& 2.12513	& 1.00120\\
\hline
$5\cdot 10^{11}$	& 151981775550	& 3236429002	& 2.12948	& 1.00116\\
\hline
$6\cdot 10^{11}$	& 182378130683	& 3889995513	& 2.13293	& 1.00112\\
\hline
$7\cdot 10^{11}$	& 212774486110	& 4544337515	& 2.13575	& 1.00109\\
\hline
$8\cdot 10^{11}$	& 243170840635	& 5199342505	& 2.13814	& 1.00107\\
\hline
$9\cdot 10^{11}$	& 273567195607	& 5854902775	& 2.14021	& 1.00105\\
\hline
$10^{12}$	& 303963550712	& 6510933430	& 2.14201	& 1.00103\\
\hline
$2^{40}$	& 334211458670	& 7164219493	& 2.14362	& 1.00101\\
\hline
\end{tabular}
\end{table}

\begin{table}[htb]
\centering
\caption{Counts of class numbers divisible by $l$}
\label{tab:CL_pl1}
\begin{tabular}{| r || r | r | r | r | r |}
\hline
\multicolumn{1}{|c||}{$x$}	& \multicolumn{1}{c|}{$3\,|\,h$}	& \multicolumn{1}{c|}{$5\,|\,h$}	& \multicolumn{1}{c|}{$7\,|\,h$}	& \multicolumn{1}{c|}{$11\,|\,h$}	& \multicolumn{1}{c|}{$13\,|\,h$}\\
\hline
\hline
$10^{11}$	& 13206088529	& 7271547905	& 4956628127	& 3011896994	& 2516050182\\
\hline
$2\cdot 10^{11}$	& 26447989308	& 14547903930	& 9914941601	& 6025009729	& 5032948550\\
\hline
$3\cdot 10^{11}$	& 39700741936	& 21825546084	& 14873726078	& 9038458883	& 7550137579\\
\hline
$4\cdot 10^{11}$	& 52959934649	& 29103662856	& 19832681021	& 12052003780	& 10067454468\\
\hline
$5\cdot 10^{11}$	& 66223739128	& 36382211005	& 24791661364	& 15065606774	& 12584840703\\
\hline
$6\cdot 10^{11}$	& 79491008890	& 43661126382	& 29750874514	& 18079320114	& 15102137499\\
\hline
$7\cdot 10^{11}$	& 92761033879	& 50940277442	& 34710302571	& 21092999797	& 17619561852\\
\hline
$8\cdot 10^{11}$	& 106033521908	& 58219944093	& 39669843978	& 24106720004	& 20137035912\\
\hline
$9\cdot 10^{11}$	& 119308020675	& 65499671827	& 44629193028	& 27120432707	& 22654498276\\
\hline
$10^{12}$	& 132584350621	& 72779583545	& 49588756987	& 30134192653	& 25171929972\\
\hline
$2^{40}$	& 145797270882	& 80023955398	& 54524158518	& 33133247297	& 27677104824\\
\hline
\end{tabular}
\end{table}

\begin{table}[htb]
\centering
\caption{Counts of class groups with $l$-rank = $r$}
\label{tab:CL_plr}
\begin{tabular}{| r || r | r | r || H H r | r | r |}
\hline
 & 
\multicolumn{3}{c||}{$r=2$} & 
\multicolumn{5}{c|}{$r=3$} \\

\multicolumn{1}{|c||}{$x$} & 
\multicolumn{1}{c|}{$l=3$} &
\multicolumn{1}{c|}{$l=5$} &
\multicolumn{1}{c||}{$l=7$} &
\multicolumn{3}{c|}{$l=3$} &
\multicolumn{1}{c|}{$l=5$} & 
\multicolumn{1}{c|}{$l=7$} \\
\hline
\hline
$10^{11}$	& 554992183	& 61905528	& 14909598	& 2278429	& 1146493	& 1891327	& 19701	& 824\\
\hline
$2\cdot 10^{11}$	& 1119549000	& 124086783	& 29864434	& 4566277	& 2298416	& 3941440	& 39455	& 1699\\
\hline
$3\cdot 10^{11}$	& 1686937952	& 186346310	& 44837690	& 6858856	& 3452213	& 6041677	& 59455	& 2555\\
\hline
$4\cdot 10^{11}$	& 2256067209	& 248638170	& 59813385	& 9153290	& 4609915	& 8170672	& 79056	& 3392\\
\hline
$5\cdot 10^{11}$	& 2826419025	& 310963856	& 74791724	& 11448254	& 5763011	& 10319592	& 99020	& 4302\\
\hline
$6\cdot 10^{11}$	& 3397716149	& 373303706	& 89772515	& 13746350	& 6918844	& 12486498	& 119058	& 5158\\
\hline
$7\cdot 10^{11}$	& 3969781768	& 435637308	& 104762170	& 16040965	& 8073267	& 14667860	& 138969	& 6050\\
\hline
$8\cdot 10^{11}$	& 4542454057	& 498010970	& 119754407	& 18334765	& 9229892	& 16861780	& 158992	& 6949\\
\hline
$9\cdot 10^{11}$	& 5115675246	& 560398913	& 134735076	& 20629307	& 10383578	& 19064061	& 179000	& 7804\\
\hline
$10^{12}$	& 5689326792	& 622806579	& 149727575	& 22924492	& 11538014	& 21274374	& 199005	& 8677\\
\hline
$2^{40}$	& 6260628955	& 684906543	& 164647966	& 25208614	& 12687370	& 23481723	& 218977	& 9586\\
\hline
\end{tabular}
\end{table}

\subsection{First Occurrences of Non-cyclic $p$-Sylow Subgroups}
During our computations, we also looked at the problem of finding
$Cl_\Delta$ with the smallest $|\Delta|$ which corresponds to a
certain $p$-group structure. This question was explored by Buell in
\cite{buell}, where he tabulated the first occurrences of what he
called ``exotic'' groups. He gave a list of first even and odd
$\Delta$, as well as the total number of them up to $2.2\cdot
10^9$. This list was extended by Ramachandran to $2\cdot 10^{11}$
\cite{ramachandran}. In Tables \ref{tab:NC_2_p}, \ref{tab:NC_3_p} and
\ref{tab:NC_4_p} we further extend Ramachandran's results by listing
first occurrences of class groups that are not present in Tables~5.13,
5.15 and 5.17 of \cite{ramachandran}.  Previously unknown minimal
discriminants whose class groups have a variety of exotic structures
were discovered, including $\Delta = -824746962451$ which is the
smallest discriminant in absolute value with $17$-rank equal to three.

We also looked at the first occurrences of doubly and trebly
non-cyclic class groups. One of the most interesting discoveries is
$\Delta = -658234953151$ with $Cl_\Delta^* \cong C(5\cdot 7\cdot 17)
\times C(5 \cdot 7 \cdot 17)$, where $C(x)$ denotes the cyclic group
of order $x$. In Tables \ref{tab:NC_mult_2} and \ref{tab:NC_mult_3},
we list first occurrences of doubly and trebly non-cyclic $p$-groups
that are not present in Tables 5.18 and 5.19 of \cite{ramachandran}.

The complete tables with all frequency counts for discriminants
satisfying \mbox{$|\Delta| < 2^{40}$} can be found in \cite{mosunov1}.
The data is soon to appear online on The $L$-functions and Modular
Forms Database \cite{lmfdb}.

\begin{longtable}[b]{| c | c | c || r | r | r | r |}
\caption{Non-cyclic rank 2 $p$-Sylow subgroups\label{tab:NC_2_p}}\\

\hline
$p$ & $e_1$	& $e_2$	& \multicolumn{1}{c|}{First even $|\Delta|$}		& \multicolumn{1}{c|}{\# even $\Delta$}	& \multicolumn{1}{c|}{First odd $|\Delta|$}	& \multicolumn{1}{c|}{\# odd $\Delta$}\\
\hline
\hline
\endfirsthead

\multicolumn{7}{c}{\normalsize{{\tablename~\thetable{}} -- continued from previous page}} \\
\hline
$p$ & $e_1$	& $e_2$	& \multicolumn{1}{c|}{First even $|\Delta|$}		& \multicolumn{1}{c|}{\# even $\Delta$}	& \multicolumn{1}{c|}{First odd $|\Delta|$}	& \multicolumn{1}{c|}{\# odd $\Delta$}\\
\hline
\hline
\endhead

\hline
\endfoot

\hline
\endlastfoot

3	& 7	& 5	& *	& *	& 253237383431	& 2\\
\hline
3	& 8	& 4	& *	& *	& 225796561799	& 10\\
\hline
3	& 10	& 2	& 1018482429656	& 2	& 65798421911	& 908\\
\hline
3	& 10	& 3	& *	& *	& 766483839959	& 2\\
\hline
3	& 11	& 1	& 786365476244	& 16	& 52623967679	& 7879\\
\hline
3	& 11	& 2	& *	& *	& 677250946319	& 24\\
\hline
3	& 12	& 1	& *	& *	& 512068796879	& 177\\
\hline
5	& 5	& 3	& *	& *	& 213265691687	& 15\\
\hline
5	& 6	& 2	& 775319038196	& 5	& 75913193999	& 175\\
\hline
5	& 7	& 1	& 573881434136	& 107	& 48662190359	& 4626\\
\hline
5	& 8	& 1	& *	& *	& 941197327199	& 3\\
\hline
7	& 3	& 3	& 798957687128	& 2	& 40111506371	& 10\\
\hline
7	& 5	& 2	& *	& *	& 336699684383	& 5\\
\hline
11	& 3	& 2	& 344379903284	& 5	& 91355041631	& 29\\
\hline
11	& 5	& 1	& *	& *	& 935094698711	& 2\\
\hline
13	& 3	& 2	& *	& *	& 366445322799	& 2\\
\hline
13	& 4	& 1	& 604812537944	& 15	& 55385334839	& 522\\
\hline
17	& 2	& 2	& 522715590248	& 3	& 94733724779	& 12\\
\hline
17	& 4	& 1	& *	& *	& 607531396391	& 7\\
\hline
23	& 3	& 1	& 428918887976	& 12	& 74447537447	& 296\\
\hline
29	& 3	& 1	& *	& *	& 323459074199	& 19\\
\hline
31	& 3	& 1	& *	& *	& 503905534439	& 14\\
\hline
53	& 2	& 1	& 313806056276	& 24	& 34862413351	& 200\\
\hline
59	& 2	& 1	& 278155567784	& 6	& 65887828631	& 81\\
\hline
61	& 2	& 1	& 388888967156	& 6	& 148712371111	& 62\\
\hline
67	& 2	& 1	& 323124297044	& 3	& 131240605511	& 28\\
\hline
73	& 2	& 1	& *	& *	& 350771311831	& 17\\
\hline
83	& 2	& 1	& *	& *	& 589364144599	& 3\\
\hline
89	& 2	& 1	& *	& *	& 619130566127	& 2\\
\hline
97	& 2	& 1	& *	& *	& 438994809599	& 2\\
\hline
101	& 2	& 1	& *	& *	& 981198752759	& 1\\
\hline
223	& 1	& 1	& 229260698804	& 17	& 36799898071	& 49\\
\hline
227	& 1	& 1	& 248738329160	& 17	& 129251563279	& 43\\
\hline
241	& 1	& 1	& 275897077784	& 13	& 74882513855	& 33\\
\hline
251	& 1	& 1	& 274131019432	& 7	& 78181110431	& 24\\
\hline
263	& 1	& 1	& 482147329592	& 7	& 37893813311	& 31\\
\hline
269	& 1	& 1	& 241103392196	& 4	& 11129396567	& 22\\
\hline
271	& 1	& 1	& 291445797352	& 5	& 171753801031	& 18\\
\hline
277	& 1	& 1	& 266610558308	& 6	& 128621435167	& 18\\
\hline
281	& 1	& 1	& 644634989492	& 2	& 266379885935	& 13\\
\hline
293	& 1	& 1	& 874615243688	& 3	& 158602460567	& 17\\
\hline
307	& 1	& 1	& 749662659128	& 3	& 149654057447	& 13\\
\hline
311	& 1	& 1	& 666221368184	& 4	& 111301462879	& 8\\
\hline
313	& 1	& 1	& 416363928728	& 3	& 303265490831	& 14\\
\hline
331	& 1	& 1	& 158739065384	& 3	& 388995885319	& 7\\
\hline
337	& 1	& 1	& 506841655124	& 2	& 283026340679	& 6\\
\hline
349	& 1	& 1	& 804641768168	& 1	& 32819826815	& 5\\
\hline
353	& 1	& 1	& 839537284648	& 2	& 305328598259	& 9\\
\hline
359	& 1	& 1	& *	& *	& 627072510479	& 4\\
\hline
373	& 1	& 1	& *	& *	& 215425181891	& 5\\
\hline
379	& 1	& 1	& *	& *	& 356510006687	& 4\\
\hline
383	& 1	& 1	& 878382375224	& 1	& 137740312007	& 6\\
\hline
397	& 1	& 1	& *	& *	& 434530437127	& 2\\
\hline
409	& 1	& 1	& *	& *	& 594857692087	& 1\\
\hline
421	& 1	& 1	& *	& *	& 422660888879	& 4\\
\hline
431	& 1	& 1	& *	& *	& 567134500223	& 3\\
\hline
433	& 1	& 1	& *	& *	& 791181108079	& 2\\
\hline
439	& 1	& 1	& *	& *	& 782761871063	& 2\\
\hline
443	& 1	& 1	& 462953812184	& 2	& 146805555551	& 2\\
\hline
449	& 1	& 1	& *	& *	& 347760731679	& 3\\
\hline
457	& 1	& 1	& 212262246356	& 1	& 413877350951	& 2\\
\hline
461	& 1	& 1	& *	& *	& 353455619411	& 4\\
\hline
463	& 1	& 1	& 1047876328724	& 1	& 679010903567	& 1\\
\hline
467	& 1	& 1	& 683325795752	& 1	& 817093587359	& 1\\
\hline
503	& 1	& 1	& *	& *	& 544027580079	& 1\\
\hline
521	& 1	& 1	& 969683875304	& 1	& 363850136623	& 1\\
\hline
577	& 1	& 1	& *	& *	& 733117084823	& 1\\
\hline
719	& 1	& 1	& *	& *	& 737463696271	& 1\\
\end{longtable}

\begin{longtable}[c]{| c | c | c | c || r | r | r | r |}
\caption{Non-cyclic rank 3 $p$-Sylow subgroups\label{tab:NC_3_p}}\\

\hline
$p$ & $e_1$	& $e_2$	& $e_3$	& \multicolumn{1}{c|}{First even $|\Delta|$}	& \multicolumn{1}{c|}{\# even $\Delta$}	& \multicolumn{1}{c|}{First odd $|\Delta|$}	& \multicolumn{1}{c|}{\# odd $\Delta$}\\
\hline
\hline
\endfirsthead

\multicolumn{8}{c}{\normalsize{{\tablename~\thetable{}} -- continued from previous page}} \\
\hline
$p$ & $e_1$	& $e_2$	& $e_3$	& \multicolumn{1}{c|}{First even $|\Delta|$}	& \multicolumn{1}{c|}{\# even $\Delta$}	& \multicolumn{1}{c|}{First odd $|\Delta|$}	& \multicolumn{1}{c|}{\# odd $\Delta$}\\
\hline
\hline
\endhead

\hline
\endfoot

\hline
\endlastfoot

3	& 3	& 3	& 2	& 341946799364	& 2	& 20687610651	& 11\\
\hline
3	& 4	& 3	& 2	& 295863285976	& 3	& 744853350587	& 1\\
\hline
3	& 5	& 2	& 2	& 412703787940	& 9	& 45248632247	& 24\\
\hline
3	& 5	& 4	& 1	& 186447381556	& 3	& 376544421947	& 5\\
\hline
3	& 6	& 2	& 2	& 582608055992	& 1	& 9483757583	& 9\\
\hline
3	& 6	& 4	& 1	& 900453600692	& 1	& 276331426207	& 2\\
\hline
3	& 7	& 2	& 2	& *	& *	& 484468933679	& 1\\
\hline
3	& 7	& 3	& 1	& 1076743681124	& 1	& 338926563823	& 10\\
\hline
3	& 8	& 2	& 1	& 276573602516	& 12	& 59714529551	& 139\\
\hline
3	& 9	& 1	& 1	& 182514096404	& 127	& 12792023879	& 978\\
\hline
3	& 9	& 2	& 1	& *	& *	& 581116399159	& 14\\
\hline
3	& 10	& 1	& 1	& 989021051864	& 1	& 146114436719	& 104\\
\hline
3	& 11	& 1	& 1	& *	& *	& 797107037711	& 1\\
\hline
5	& 4	& 2	& 1	& 204195796664	& 3	& 116279191211	& 7\\
\hline
5	& 6	& 1	& 1	& *	& *	& 349008665407	& 5\\
\hline
7	& 2	& 2	& 1	& 439240920004	& 1	& 868770849819	& 3\\
\hline
7	& 4	& 1	& 1	& 356820088964	& 1	& 451900165735	& 4\\
\hline
11	& 2	& 1	& 1	& 889484965924	& 2	& 145931588651	& 9\\
\hline
13	& 1	& 1	& 1	& 218639119912	& 11	& 38630907167	& 20\\
\hline
17	& 1	& 1	& 1	& *	& *	& 824746962451	& 2\\
\end{longtable}

\begin{longtable}[c]{| c | c | c | c | c || r | r | r | r |}
\caption{Non-cyclic rank 4 $p$-Sylow subgroups\label{tab:NC_4_p}}\\

\hline
$p$ & $e_1$	& $e_2$	& $e_3$	& $e_4$	& \multicolumn{1}{c|}{First even $|\Delta|$}		& \multicolumn{1}{c|}{\# even $\Delta$}	& \multicolumn{1}{c|}{First odd $|\Delta|$}	& \multicolumn{1}{c|}{\# odd $\Delta$}\\
\hline
\hline
\endfirsthead

\multicolumn{9}{c}{\normalsize{{\tablename~\thetable{}} -- continued from previous page}} \\
\hline
$p$ & $e_1$	& $e_2$	& $e_3$	& $e_4$	& \multicolumn{1}{c|}{First even $|\Delta|$}		& \multicolumn{1}{c|}{\# even $\Delta$}	& \multicolumn{1}{c|}{First odd $|\Delta|$}	& \multicolumn{1}{c|}{\# odd $\Delta$}\\
\hline
\hline
\endhead

\hline
\endfoot

\hline
\endlastfoot

3	& 3	& 3	& 1	& 1	& *	& *	& 1074734433547	& 1\\
\hline
3	& 4	& 2	& 1	& 1	& 426126877012	& 3	& 128589208863	& 8\\
\hline
3	& 5	& 2	& 1	& 1	& *	& *	& 473827747963	& 2\\
\hline
3	& 6	& 1	& 1	& 1	& 460093393912	& 8	& 76951070303	& 15\\
\hline
3	& 7	& 1	& 1	& 1	& 1047320556596	& 1	& 513092626699	& 2\\
\hline
3	& 8	& 1	& 1	& 1	& *	& *	& 226138531999	& 2\\
\end{longtable}

\begin{longtable}[c]{| c | c || r | r | r | r |}
\caption{Doubly non-cyclic class groups\label{tab:NC_mult_2}}\\

\hline
$p_1$	& $p_2$	& \multicolumn{1}{c|}{First even $|\Delta|$}	& \multicolumn{1}{c|}{\# even $\Delta$}	& \multicolumn{1}{c|}{First odd $|\Delta|$}	& \multicolumn{1}{c|}{\# odd $\Delta$}\\
\hline
\hline
\endfirsthead

\multicolumn{6}{c}{\normalsize{{\tablename~\thetable{}} -- continued from previous page}} \\
\hline
$p_1$	& $p_2$	& \multicolumn{1}{c|}{First even $|\Delta|$}	& \multicolumn{1}{c|}{\# even $\Delta$}	& \multicolumn{1}{c|}{First odd $|\Delta|$}	& \multicolumn{1}{c|}{\# odd $\Delta$}\\
\hline
\hline
\endhead

\hline
\endfoot

\hline
\endlastfoot

3	& 83	& {411040250696}	& 8	& 50476998239	& 69\\
\hline
3	& 89	& {271776528392}	& 14	& 146604777199	& 29\\
\hline
3	& 97	& {373716927704}	& 5	& 43344787079	& 22\\
\hline
3	& 101	& {204919229864}	& 3	& {270845549231}	& 12\\
\hline
3	& 103	& {374301791476}	& 8	& 93069031703	& 14\\
\hline
3	& 107	& {747657517988}	& 2	& 193384461719	& 15\\
\hline
3	& 109	& {379370724596}	& 5	& 35029686023	& 17\\
\hline
3	& 127	& {761263140536}	& 1	& 127466536019	& 10\\
\hline
3	& 131	& *	& *	& {248486020319}	& 2\\
\hline
3	& 137	& *	& *	& {373309196719}	& 4\\
\hline
3	& 139	& *	& *	& {261265037799}	& 3\\
\hline
3	& 149	& *	& *	& {555574557467}	& 4\\
\hline
3	& 157	& *	& *	& {258504106919}	& 2\\
\hline
3	& 163	& *	& *	& {288700332223}	& 5\\
\hline
3	& 191	& *	& *	& {778133573263}	& 1\\
\hline
3	& 193	& *	& *	& 4 {15837893871}	& 1\\
\hline
3	& 197	& *	& *	&  {675588676571}	& 1\\
\hline
3	& 223	& *	& *	&  {1044678632711}	& 1\\
\hline
5	& 47	&  {337410526616}	& 16	& 8182208159	& 78\\
\hline
5	& 53	&  {375201391636}	& 8	& 22759605719	& 28\\
\hline
5	& 59	&  {842452697976}	& 2	& 166413410411	& 20\\
\hline
5	& 61	&  {621148062232}	& 2	& 198540663599	& 14\\
\hline
5	& 67	&  {952877473160}	& 1	& { 202658297511}	& 13\\
\hline
5	& 79	& *	& *	&  {695299489415}	& 4\\
\hline
5	& 83	& *	& *	&  {255558978287}	& 5\\
\hline
5	& 97	& *	& *	&  {957408127639}	& 1\\
\hline
5	& 107	& *	& *	&  {895542638663}	& 1\\
\hline
7	& 37	&  {220308406520}	& 5	& 49918973471	& 36\\
\hline
7	& 43	&  {395768104936}	& 1	& 57006644887	& 18\\
\hline
7	& 47	&  {611628524996}	& 2	& 98533572251	& 12\\
\hline
7	& 53	&  {819974042456}	& 1	&  {532593252151}	& 6\\
\hline
7	& 59	& *	& *	&  {746029216663}	& 3\\
\hline
7	& 61	& *	& *	&  {530458082031}	& 2\\
\hline
7	& 79	& *	& *	&  {1010896284767}	& 1\\
\hline
7	& 101	& *	& *	&  {613532171711}	& 1\\
\hline
11	& 19	&  {293745669956}	& 33	& 19439678123	& 86\\
\hline
11	& 23	&  {440245788692}	& 7	& 94266055451	& 45\\
\hline
11	& 29	&  {258828614756}	& 2	&  {246806029679}	& 13\\
\hline
11	& 31	&  {752299766228}	& 1	& 167546860535	& 6\\
\hline
11	& 37	& *	& *	&  {507297592171}	& 1\\
\hline
13	& 23	&  {886308340568}	& 1	&  {303087341987}	& 15\\
\hline
13	& 31	&  {1042065325544}	& 1	&  {309693265351}	& 5\\
\hline
13	& 37	& *	& *	&  {583833769207}	& 1\\
\hline
13	& 41	&  {969016080404}	& 1	&  {407911409771}	& 2\\
\hline
17	& 19	& 150334566104	& 2	&  {473841789911}	& 9\\
\hline
17	& 23	&  {432363302164}	& 2	& 54134972891	& 3\\
\hline
17	& 29	& *	& *	&  {892052200651}	& 2\\
\hline
17	& 31	& *	& *	&  {1035367542059}	& 1\\
\hline
19	& 23	& *	& *	&  {659380117199}	& 1\\
\hline
19	& 29	& *	& *	&  {915336787039}	& 1\\
\end{longtable}

\begin{longtable}[c]{| c | c | c || r | r | r | r |}
\caption{Trebly non-cyclic class groups\label{tab:NC_mult_3}}\\

\hline
$p_1$	& $p_2$	& $p_3$	& \multicolumn{1}{c|}{First even $|\Delta|$}	& \multicolumn{1}{c|}{\# even $\Delta$}	& \multicolumn{1}{c|}{First odd $|\Delta|$}	& \multicolumn{1}{c|}{\# odd $\Delta$}\\
\hline
\hline
\endfirsthead

\multicolumn{7}{c}{\normalsize{{\tablename~\thetable{}} -- continued from previous page}} \\
\hline
$p_1$	& $p_2$	& $p_3$	& \multicolumn{1}{c|}{First even $|\Delta|$}	& \multicolumn{1}{c|}{\# even $\Delta$}	& \multicolumn{1}{c|}{First odd $|\Delta|$}	& \multicolumn{1}{c|}{\# odd $\Delta$}\\
\hline
\hline
\endhead

\hline
\endfoot

\hline
\endlastfoot

3	& 5	& 17	& {278849168408}	& 19	& 60235736039	& 63\\
\hline
3	& 5	& 23	& {703386940456}	& 3	& 148439200263	& 14\\
\hline
3	& 5	& 29	& *	& *	& {300193517399}	& 5\\
\hline
3	& 5	& 31	& *	& *	& {323714678543}	& 5\\
\hline
3	& 5	& 37	& *	& *	& {999098015071}	& 1\\
\hline
3	& 7	& 23	& *	& *	& {805192394183}	& 1\\
\hline
5	& 7	& 11	& *	& *	& {656450533751}	& 6\\
\hline
5	& 7	& 13	& {786460186856}	& 1	& 110671542299	& 3\\
\hline
5	& 7	& 17	& *	& *	& {658234953151}	& 1\\
\end{longtable}

\subsection{Euler's Conjecture on Idoneal Numbers}
Consider a discriminant $\Delta$ such that $Cl_\Delta$ is isomorphic
to $(\mathbb Z/2\mathbb Z)^l$ for some $l > 0$. All such $\Delta$ are
related to so-called \emph{idoneal} numbers, which were studied by
Euler and Gau\ss\,(see the extensive survey on idoneal numbers by Kani
\cite{kani}). A positive number $D$ is idoneal if every integer $n$,
which is uniquely representable in the form $n = x^2 \pm Dy^2$ with
$\gcd(x^2,Dy^2) = 1$, is either a prime, or a prime power, or twice
one of these.  Both Euler and Gau\ss\, tabulated idoneal numbers, and
conjectured that the largest of them does not exceed 1848
\cite[\textsection 303]{gauss}. From the class group perspective, it
means that $\Delta = -5460$ is the largest fundamental discriminant
such that $Cl_\Delta \cong (\mathbb Z/2\mathbb Z)^l$. In 1918, the
hypothesis of Euler and Gau\ss\, was confirmed by Hecke and Landau
under the assumption of the ERH \cite{landau}. However,
unconditionally this problem still remains open, though Weinberger was
able to prove that there exists \emph{at most} one idoneal number
exceeding 1848 \cite{weinberger}. In our computations, we confirm that
up to $2^{40}$ the largest in its absolute value fundamental
discriminant $\Delta$ with $Cl_\Delta \cong (\mathbb Z/2\mathbb Z)^l$
is $\Delta = -5460$. This result agrees with findings of Euler and
Gau\ss. Note that there exists one non-fundamental discriminant, namely
$\Delta = -7392$, which is larger than $-5460$ in its absolute value and
has the group structure as above.

\section{Further Work} \label{sec:furtherwork}

Our novel approach to class group tabulation has enabled us to extend
the feasibility limit.  Pushing our methods further would probably
require a class number tabulation mechanism for $\Delta \equiv 1
\pmod{8}$. Presently, no efficient class number tabulation formulas
are known for this congruence class. One formula that might be of
interest for future exploration is due to Humbert \cite[Section
  6]{watson}, who discovered that
\begin{equation} \label{eq:humbert}
\sum_{n = 0}^\infty F(8n+7)q^n = S^{-1}(q)\sum_{n=1}^\infty(-1)^{n+1}n^2\frac{q^{\frac{n(n+1)}{2}}-1}{1+q^n},
\end{equation} 
where
$$
S(q) = \sum_{n = 0}^\infty(-1)^n(2n+1)q^{\frac{n(n+1)}{2}} = 1 - 3q + 5q^3 - 7q^6 + 9q^{10} - \ldots\,\,\,.
$$
Despite its look, the large series on the right hand side of
(\ref{eq:humbert}) does not take long to initialize, as it is simple
to derive the formula for its $n$-th coefficient. The main problem
lies in the computation of $S^{-1}(q)$, which can take a significant
amount of time, especially if $S(q)$ is of a high degree. Also, the
coefficients of $S(q)$ grow very fast. For example, its 16-th
coefficient has bit size 10, 64-th --- bit size 59, and 65536-th fits
into 1135 bits. These coefficients have to be somehow truncated, for
example, by reducing them modulo some prime $p$, such that $F(8n+7) <
p$ for any $n < (N-7)/8$. However, this approach also brings certain
difficulties, as it significantly increases the bit size parameter
$s$. Despite all the obstacles, the usage of the formula
(\ref{eq:humbert}) might still be faster than the conditional
computation of $\Delta \equiv 1 \pmod{8}$ followed by the verification
procedure. We tried to use this approach, and in fact
our implementation includes a subroutine \texttt{invert} for
out-of-core polynomial inversion \cite{mosunov2}. This subroutine
utilizes a Newton iteration algorithm \cite[Algorithm 9.3]{gathen},
which performs the inversion of an arbitrary polynomial to degree $2^n
- 1$ by sequentially computing its inverse to degrees $3, 7, \ldots,
2^k - 1, \ldots, 2^n - 1$ for $2 \leq k \leq n$. Each iteration in
this algorithm requires one squaring of a polynomial and one
polynomial multiplication. Unfortunately, we were unable to produce
class numbers using this method due to the number of difficulties
previously mentioned.

We also believe that the class group computation can get accelerated
by using Sutherland's $p$-group discrete logarithm algorithms
\cite{sutherland}. The idea is simple: when the class number
$h(\Delta)$ is known, instead of computing all of the potentially
non-cyclic subgroups of $Cl_\Delta$ we compute the structure of each
potentially non-cyclic $p$-group separately. Sutherland's algorithms
may be especially useful when resolving the structure of a 2-group, as
we can precompute its rank by factoring $\Delta$. In some cases, the
2-rank allows us to terminate the 2-group resolution earlier by
ignoring some of its generators of order 2.

Finally, we note that the question of unconditional tabulation of
class groups with positive $\Delta$ is still left open, and is
currently work in progress.

\bibliographystyle{amsplain}

\begin{thebibliography}{10}

\bibitem[Bac90]{bach1}
E. Bach,
\emph{Explicit bounds for primality testing and related problems},
Mathematics of Computation 55, 1990, pp. 355 -- 380.

\bibitem[Bac95]{bach2}
E. Bach,
\emph{Improved approximations for {E}uler products},
Number Theory: CMS Proceedings 15, American Mathematical Society Providence, RI, 1995, pp. 13 -- 28.


\bibitem[Bel24]{bell}
E.~T. Bell,
\emph{The class number relations implicit in the {D}isquisitiones arithmeticae},
Bulletin of the American Mathematical Society 30 (5-6), 1924, pp. 236 -- 238.


\bibitem[BJT97]{buchmann}
J. Buchmann, M.~J. Jacobson, Jr., E. Teske,
\emph{On some computational problems in finite abelian groups},
Mathematics of Computation 66 (220), 1997, pp. 1663 -- 1687.

\bibitem[BM74]{borodin}
A. Borodin, R.~T. Moenck,
\emph{Fast modular transforms},
Journal of Computer and System Sciences 8, 1974, pp. 366 -- 386.

\bibitem[Bue99]{buell}
D.~A. Buell,
\emph{The last exhaustive computation of class groups of complex quadratic number fields},
Number theory, CRM Proceedings and Lecture Notes 19, American Mathematical Society, Providence, RI, 1999, pp. 35 -- 53.

\bibitem[GG03]{gathen}
J. von zur Gathen, J. Gerhard,
\emph{Modern Computer Algebra},
Cambridge University Press, 2nd edition, 2003.


\bibitem[CL84]{cohen2}
H. Cohen, H.~W. Lenstra, Jr.,
\emph{Heuristics on class groups},
Number Theory, Noordwijkerhout, Lecture Notes in Mathematics 1052, Springer-Verlag, New York, 1984, pp. 26 -- 36.

\bibitem[Coh93]{cohen3}
H. Cohen,
\emph{A course in computational algebraic number theory},
Springer-Verlag, Berlin, 1993.

\bibitem[Gau86]{gauss}
C. F. Gau\ss,
\emph{Disquisitiones arithmeticae},
1798. Translated into English by A.~A. Clarke and reprinted, Springer-Verlag, 1986.

\bibitem[Gro85]{grosswald}
E. Grosswald,
\emph{Representation of Integers as Sums of Squares},
Springer-Verlag, New York, 1985.

\bibitem[Har14]{flint}
W.~B. Hart,
\emph{FLINT: Fast Library for Number Theory}, version 2.4.1, \url{http://www.flintlib.org}, 2014.

\bibitem[HMcC89]{hafner}
J. Hafner, K. McCurley,
\emph{A rigorous subexponential algorithm for computation of class groups},
Journal of American Mathematical Society 2, 1989, pp. 837 -- 850.

\bibitem[HTW10]{hart}
W.~B. Hart, G. Tornar\'ia, M. Watkins,
\emph{Congruent number theta coefficients to $10^{12}$},
Algorithmic Number Theory - ANTS-IX (Nancy, France), Lecture Notes in Computer Science 6197, Springer-Verlag, Berlin, 2010, pp. 186 -- 200.


\bibitem[JRW06]{jacobson}
M.~J. Jacobson, Jr., S. Ramachandran, H. C. Williams,
\emph{Numerical results on class groups of imaginary quadratic fields},
Algorithmic Number Theory - ANTS-VII (Berlin, Germany), Lecture Notes in Computer Science 4076, Springer-Verlag, Berlin, 2006, pp. 87 -- 101.

\bibitem[Kan11]{kani}
E. Kani,
\emph{Idoneal numbers and some generalizations},
Annales des Sciences Math\'ematiques du Qu\'ebec 35, 2011, pp. 197 -- 227.

\bibitem[Kro60]{kronecker}
L. Kronecker,
\emph{\"Uber die Anzahl der verschiedenen classen quadratischer Formen von negativer Determinante},
Journal f\"ur die reine und angewandte Mathematik 57 (4), 1860, pp. 248 -- 255.

\bibitem[Lan18]{landau}
E. Landau,
\emph{\"Uber die Klassenzahl imagin\"ar-quadratischer Zahlk\"orper},
Ges. Wiss. G\"ottingen, Math.-Phys., 1918, pp. 285 -- 295.

\bibitem[Lit28]{littlewood}
J.~E. Littlewood,
\emph{On the class number of the corpus $P(\sqrt{-k})$},
Proceedings of the London Mathematical Society 27, 1928, pp. 358 -- 372.

\bibitem[LMFDB]{lmfdb}
The LMFDB Collaboration,
The $L$-functions and Modular Forms Database,
\url{http://www.lmfdb.org}, 2015.

\bibitem[McC89]{mccurley}
K. McCurley,
\emph{Cryptographic key distribution and computation in class groups},
Proceedings of NATO ASI Number Theory and Applications, Kluwer Academic Publishers, 1989, pp. 459 -- 479.

\bibitem[Mey12]{meyer}
D. Meyer,
\emph{The second $p$-class group of a number field},
International Journal of Number Theory 2 (8), 2012, pp. 471 -- 505.

\bibitem[Mos14a]{mosunov1}
A.~S. Mosunov,
\emph{Unconditional Class Group Tabulation to $2^{40}$},
Master's thesis, University of Calgary, Calgary, Alberta, 2014.

\bibitem[Mos14b]{mosunov2}
A.~S. Mosunov,
\emph{Class group tabulation program},
\url{https://github.com/amosunov}, 2014.

\bibitem[Par14]{pari}
The PARI~Group,
\emph{PARI/GP version \texttt{2.7.1}},
\url{http://pari.math.u-bordeaux.fr}, Bordeaux, 2014.

\bibitem[Ram01]{ramare}
O. Ramar\'e,
\emph{Approximate formulae for $L(1,\chi)$},
Acta Arithmetica 100, 2001, pp. 245 -- 266.

\bibitem[Ram06]{ramachandran}
S. Ramachandran,
\emph{Class groups of quadratic fields},
Master's thesis, University of Calgary, Calgary, Alberta, 2006.

\bibitem[Say13a]{sayles1}
M. Sayles,
\emph{Improved arithmetic in the ideal class group of imaginary quadratic fields with an application to integer factoring},
Master's thesis, University of Calgary, Calgary, Alberta, 2013.

\bibitem[Say13b]{sayles2}
M. Sayles,
C libraries \texttt{optarith} and \texttt{qform} for fast binary quadratic form arithmetic,
\url{https://github.com/maxwellsayles}, 2013.

\bibitem[Sha73]{shanks}
\emph{Systematic examination of Littlewood's bounds on $L(1,\chi)$},
Proceedings of Symposia in Pure Mathematics, AMS, Providence, R.I., 1973, pp. 267 -- 283.

\bibitem[Sut11]{sutherland}
A.~V. Sutherland,
\emph{Structure computation and discrete logarithms in finite abelian $p$-groups},
Mathematics of Computation 80 (273), 2011, pp. 477 -- 500.

\bibitem[SvdV91]{schoof}
R. Schoof, M. van der Vlugt,
\emph{Hecke operators and the weight distributions of certain codes},
Journal of Combinatorial Theory 57, 1991, pp. 163 -- 186.


\bibitem[Wat35]{watson}
G.~N. Watson,
\emph{Generating functions of class numbers},
Compositio Mathematica, tome 1, 1935, pp. 39 -- 68.

\bibitem[Wei73]{weinberger}
P. Weinberger,
\emph{Exponents of the class groups of complex quadratic fields},
Acta Arithmetica 22, 1973, pp. 117 -- 124.

\bibitem[Wes14]{westgrid}
Hungabee specification,
\url{https://www.westgrid.ca/support/systems/Hungabee}, 2014.
\end{thebibliography}

\providecommand{\bysame}{\leavevmode\hbox to3em{\hrulefill}\thinspace}
\providecommand{\MR}{\relax\ifhmode\unskip\space\fi MR }
\providecommand{\MRhref}[2]{%
  \href{http://www.ams.org/mathscinet-getitem?mr=#1}{#2}
}
\providecommand{\href}[2]{#2}

\end{document}